\documentclass[reqno]{amsart}

\usepackage{amssymb}
\usepackage{graphicx}
\usepackage{amscd}
\usepackage[hidelinks]{hyperref}
\usepackage{color}
\usepackage{float}
\usepackage{graphics,amsmath,amssymb}
\usepackage{amsthm}
\usepackage{amsfonts}
\usepackage{latexsym}
\usepackage{epsf}
\usepackage{xifthen}
\usepackage{mathrsfs}
\usepackage{dsfont}
\usepackage{makecell}
\usepackage[FIGTOPCAP]{subfigure}
\usepackage{amsmath}
\allowdisplaybreaks[4]
\usepackage{listings}
\usepackage{etoolbox}
\usepackage{fancyhdr}
\usepackage{pdflscape}
\usepackage[title,toc,titletoc]{appendix}
\usepackage{enumitem}
\usepackage[noadjust]{cite}

 \newtheoremstyle{mytheorem}
 {3pt}
 {3pt}
 {\slshape}
 {}
 {\bfseries}
 {.}
 { }
 {}

\numberwithin{equation}{section}

\theoremstyle{theorem}
\newtheorem{theorem}{Theorem}[section]
\newtheorem*{theorem*}{Theorem}

\newtheorem{lemma}[theorem]{Lemma}

\newtheorem{conjecture}{Conjecture}[section]

\newtheorem{innercustomgeneric}{\customgenericname}
\providecommand{\customgenericname}{}
\newcommand{\newcustomtheorem}[2]{%
	\newenvironment{#1}[1]
	{%
		\renewcommand\customgenericname{#2}%
		\renewcommand\theinnercustomgeneric{##1}%
		\innercustomgeneric
	}
	{\endinnercustomgeneric}
}
\newcustomtheorem{ctheorem}{Theorem}
\newcustomtheorem{clemma}{Lemma}

\theoremstyle{definition}

\newtheorem*{example*}{Example}
\newtheorem*{examples*}{Examples}
\newtheorem{remark}{Remark}[section]
\newtheorem*{remark*}{Remark}
\newtheorem*{remarks*}{Remarks}

\newtheoremstyle{named}{}{}{\itshape}{}{\bfseries}{.}{.5em}{#1\thmnote{ #3}}
\theoremstyle{named}

\newcommand{\Keywords}[1]{\ifthenelse{\isempty{#1}}{}{\smallskip \smallskip \noindent \textbf{Keywords}. #1}}
\newcommand{\MSC}[2][2010]{\ifthenelse{\isempty{#2}}{}{\smallskip \smallskip \noindent \textbf{#1MSC}. #2}}
\newcommand{\abstractnote}[1]{\ifthenelse{\isempty{#1}}{}{\smallskip \smallskip \noindent \textsuperscript{\dag}#1}}

\makeatletter
\def\specialsection{\@startsection{section}{1}%
  \z@{\linespacing\@plus\linespacing}{.5\linespacing}%
  {\normalfont}}
\def\section{\@startsection{section}{1}%
  \z@{.7\linespacing\@plus\linespacing}{.5\linespacing}%
  {\normalfont\scshape}}
\patchcmd{\@settitle}{\uppercasenonmath\@title}{\Large\boldmath}{}{}
\patchcmd{\@settitle}{\begin{center}}{\begin{flushleft}}{}{}
\patchcmd{\@settitle}{\end{center}}{\end{flushleft}}{}{}
\patchcmd{\@setauthors}{\MakeUppercase}{\normalsize}{}{}
\patchcmd{\@setauthors}{\centering}{\raggedright}{}{}
\patchcmd{\section}{\scshape}{\large\bfseries\boldmath}{}{}
\patchcmd{\subsection}{\bfseries}{\bfseries\boldmath}{}{}
\renewcommand{\@secnumfont}{\bfseries}
\patchcmd{\@startsection}{\@afterindenttrue}{\@afterindentfalse}{}{}
\patchcmd{\abstract}{\leftmargin3pc}{\leftmargin1pc}{}{}

\def\maketitle{\par
  \@topnum\z@ 
  \@setcopyright
  \thispagestyle{empty}
  \ifx\@empty\shortauthors \let\shortauthors\shorttitle
  \else \andify\shortauthors
  \fi
  \@maketitle@hook
  \begingroup
  \@maketitle
  \toks@\@xp{\shortauthors}\@temptokena\@xp{\shorttitle}%
  \toks4{\def\\{ \ignorespaces}}
  \edef\@tempa{%
    \@nx\markboth{\the\toks4
      \@nx\MakeUppercase{\the\toks@}}{\the\@temptokena}}%
  \@tempa
  \endgroup
  \c@footnote\z@
  \@cleartopmattertags
}
\makeatother

\newcommand{\zrie}[1]{\zeta\!\left(#1\right)}
\newcommand{\Arg}{\operatorname{Arg}}
\newcommand{\Res}{\operatorname{Res}}
\newcommand{\lcm}{\operatorname{lcm}}
\newcommand{\dg}[1]{\frac{\Gamma'}{\Gamma}\!\left(#1\right)}
\newcommand{\Main}{\mathfrak{M}}
\newcommand{\mfO}{\mathbb{O}}
\newcommand{\sR}{\mathscr{R}}
\newcommand{\sL}{\mathscr{L}}
\newcommand{\cR}{\mathcal{R}}

\title[Nonmodular infinite products]{Nonmodular infinite products and a conjecture of\\ Seo and Yee}

\author[S. Chern]{Shane Chern}
\address{Department of Mathematics and Statistics, Dalhousie University, Halifax, Nova Scotia, B3H 4R2, Canada}
\email{chenxiaohang92@gmail.com}

\date{}

\begin{document}

\maketitle

\begin{abstract}

We will tackle a conjecture of S.~Seo and A.~J.~Yee, which says that the series expansion of $1/(q,-q^3;q^4)_\infty$ has nonnegative coefficients. Our approach relies on an approximation of the generally nonmodular infinite product $1/(q^a;q^M)_\infty$, where $M$ is a positive integer and $a$ is any of $1,2,\ldots,M$.


\MSC{11P55, 11P82.}
\end{abstract}

\section{Introduction}

Throughout, we adopt the standard $q$-series notation:
\begin{align*}
(A;q)_\infty:=\prod_{k\ge 0} (1-A q^k)
\end{align*}
and
\begin{align*}
(A,B,\ldots,C;q)_\infty:=(A;q)_\infty(B;q)_\infty\cdots(C;q)_\infty.
\end{align*}

In their work on the index of seaweed algebras and integer partitions, Seo and Yee \cite{SY2019} proved that an earlier conjecture of Coll, A. Mayers and N. Mayers \cite{CMM2018} is equivalent to the following nonnegativity conjecture.

\begin{conjecture}[Seo--Yee]\label{conj-sy}
	The series expansion of
	\begin{equation}\label{eq:conj-sy}
	\frac{1}{(q,-q^3;q^4)_\infty}
	\end{equation}
	has nonnegative coefficients.
\end{conjecture}

As a $q$-hypergeometric proof of this conjecture is notoriously difficult to find, one may hope to derive an asymptotic formula for the coefficients. If one is also patient enough to compute an explicit bound of the error term, then a direct examination will yield a proof of such nonnegativity. Note that \eqref{eq:conj-sy} is not modular; in particular, it cannot be written as a quotient of Dedekind eta functions or Jacobi theta functions. Thus, a Rademacher-type exact formula is out of reach. Also, if we rewrite this product as
$$\frac{(q^3;q^4)_\infty}{(q;q^4)_\infty (q^6;q^8)_\infty},$$
then the numerator $(q^3;q^4)_\infty$ prevents us using Meinardus' powerful approach \cite{Mei1954}. One of the few works on the asymptotics of nonmodular infinite products is due to Grosswald \cite{Gro1958}, who absorbed ideas from Lehner \cite{Leh1941} and Livingood \cite{Liv1945}. In his paper, the infinite product $1/(q^a;q^M)_\infty$ 
with a prime modulus $M$ is considered. However, a closer examination of Grosswald's paper reveals several mistakes.\footnote{For example, in the calculation of the residue $R_3$ on page 119 of \cite{Gro1958}, the following term is missing:
$$-\frac{1}{4\pi k^2}\sum_{\rho,\mu,\nu=1}^k \sin \frac{2\pi \rho \mu}{k} \cos \frac{2\pi \rho h' \nu}{k}\frac{2\zrie{2,\frac{\mu}{k}}\zrie{0,\frac{\nu}{k}}}{\pi z}.$$
In other words, in this $R_3$, there is also a term of the form $z^{-1}$ times an expression.} 
It is also a natural question to ask about the case where the modulus is composite.

Let $M$ be a positive integer and $a$ be any of $1,2,\ldots,M$. The first object of this paper is to investigate the asymptotic behavior of
\begin{equation}\label{eq:Phi-def-1}
\Phi_{a,M}(q):=\log\left(\frac{1}{(q^a;q^M)_\infty}\right)
\end{equation}
when the complex variable $q$ with $|q|<1$ approaches an arbitrary root of unity.

For this purpose, we begin with some standard notations. As usual, $\zeta(s)$ and $\zeta(s,\alpha)$ denote the Riemann and Hurwitz zeta functions, respectively. We write $\zeta'(s,\alpha)$ the partial derivative of Hurwitz zeta function with respect to $s$, namely,
$$\zeta'(s,\alpha):=\frac{\partial}{\partial s}\zeta(s,\alpha).$$
Meanwhile, $\Gamma(s)$ is the Gamma function and $\gamma$ is the Euler--Mascheroni constant. We shall also adopt the conventional Vinogradov notation $f(x)\ll g(x)$, which means that there exists an absolute constant $C$ such that $|f(x)|\le C g(x)$. If the constant $C$ depends on some variables, we attach a subscript and write $f(x)\ll_{\textsf{var}} g(x)$. Finally, we write the greatest common divisor of integers $m$ and $n$ as $(m,n)$.

\begin{theorem}\label{th:main-est}
	Let $X\ge 16$ be a sufficiently large positive number. Let
	\begin{equation}
	q:=e^{-\tau+\frac{2\pi i h}{k}},
	\end{equation}
	where $1\le h\le k\le \lfloor\sqrt{2\pi X}\rfloor=:N$ with $(h,k)=1$ and $\tau:=X^{-1}+2\pi i Y$ with $|Y|\le \frac{1}{kN}$. Let $M$ be a positive integer and $a$ be any of $1,2,\ldots,M$. If we denote by $b$ the unique integer between $1$ and $(k,M)$ such that $b\equiv -ha \pmod{(k,M)}$ and write
	$$b^*:=\begin{cases}
	(k,M)-b & \text{if $b\ne (k,M)$},\\
	(k,M) & \text{if $b= (k,M)$},
	\end{cases}$$
	then
	\begin{align}\label{eq:th-main}
	\Phi_{a,M}(q)&=\frac{1}{\tau}\frac{(k,M)^2}{k^2 M}\Bigg[\pi^2\left(\frac{ b^2}{(k,M)^2}-\frac{ b}{(k,M)}+\frac{1}{6}\right)\nonumber\\
	&\quad+2\pi i\left(-\zeta'\left(-1,\frac{b}{(k,M)}\right)+\zeta'\left(-1,\frac{b^*}{(k,M)}\right)\right)\Bigg]+E,
	\end{align}
	where
	\begin{equation}
	|\Re(E)|\ll_{M} X^{\frac{1}{2}}\log X.
	\end{equation}
\end{theorem}

\begin{remark}
	One may explicitly compute an upper bound for $|\Re(E)|$ by means of the estimations in \eqref{eq:bound-integral}, \eqref{eq:R12}, \eqref{eq:R13}, \eqref{eq:R14}, \eqref{eq:R21}, \eqref{eq:R22}, \eqref{eq:R23}, \eqref{eq:R3} and \eqref{eq:R42}. It should be noted that in some of these estimations, there is an expression $(k,M)$. However, we trivially have $1\le (k,M)\le M$, so the asymptotic relation $|\Re(E)|\ll X^{\frac{1}{2}}\log X$ can be expressed in a way that depends only on the additional parameter $M$.
\end{remark}

\begin{remark}\label{rmk:Q_h/k}
	Let $\mathcal{Q}_{h/k}$ be the set of $q$ with respect to $h/k$ as defined in Theorem \ref{th:main-est}, that is,
	$$\mathcal{Q}_{h/k}:=\big\{e^{-\frac{1}{X}+2\pi i(\frac{h}{k}-Y)}\,:\, |Y|\le \tfrac{1}{kN}\big\}.$$
	For any $q$ with $|q|=e^{-\frac{1}{X}}$, we are always able to find a certain $h/k$ such that $q\in \mathcal{Q}_{h/k}$. This is a direct consequence of Dirichlet's approximation theorem, asserting that $\mathbb{R}/\mathbb{Z}$ can be covered by intervals
	$$\bigcup_{\substack{1\le h\le k\le N\\ (h,k)=1}}\left[\frac{h}{k}-\frac{1}{kN},\frac{h}{k}+\frac{1}{kN}\right].$$
\end{remark}

Equipped with Theorem \ref{th:main-est}, we prove Conjecture \ref{conj-sy} for large enough $n$.

\begin{theorem}\label{th:g(n)-asymp}
	Let
	\begin{align}\label{eq:G(q)-1}
	G(q):=\sum_{n\ge 0}g(n)q^n = \frac{1}{(q,-q^3;q^4)_\infty}.
	\end{align}
	We have, as $n\to\infty$,
	\begin{equation}\label{eq:g(n)-asymp}
	g(n)\sim \frac{\pi^{\frac{1}{4}}\Gamma(\tfrac{1}{4})}{2^{\frac{9}{4}}3^{\frac{3}{8}}n^{\frac{3}{8}}}I_{-\frac{3}{4}}\left(\frac{\pi}{2}\sqrt{\frac{n}{3}}\right)+(-1)^n\frac{\pi^{\frac{3}{4}}\Gamma(\tfrac{3}{4})}{2^{\frac{11}{4}}3^{\frac{5}{8}}n^{\frac{5}{8}}}I_{-\frac{5}{4}}\left(\frac{\pi}{2}\sqrt{\frac{n}{3}}\right),
	\end{equation}
	where $I_s(x)$ is the modified Bessel function of the first kind. Further, when $n\ge2.4\times 10^{14}$, we have $g(n)>0$.
\end{theorem}

\begin{remark}
	After this paper was finished, William Craig \cite{Cra2021} offered a full proof of the Seo--Yee Conjecture via a different approach with recourse to the Euler--Maclaurin summations.
\end{remark}

This paper is organized as follows. In Section \ref{sec:prep}, we make necessary preparations for the proof of Theorem \ref{th:main-est}. More precisely, we rewrite $\Phi_{a,M}(q)$ by means of the inverse Mellin transform so that we are led to the evaluations of certain integrals and residues, which will be done in Sections \ref{sec:integrals} and \ref{sec:res}, respectively. Then we move to the second object, namely, the Seo--Yee Conjecture. In Section \ref{sec:G-minor}, we apply Theorem \ref{th:main-est} to establish explicit bounds for $G(q)$ on the minor arcs. Further, Section \ref{sec:G-major} is devoted to precise approximations of $G(q)$ near the dominant poles. Finally, we prove Theorem \ref{th:g(n)-asymp} is Section \ref{sec:cir}.

\section{Preparations}\label{sec:prep}

Notice that
\begin{equation}\label{eq:Phi-def}
\Phi_{a,M}(q)=\log\left(\frac{1}{(q^a;q^M)_\infty}\right)=\sum_{\substack{m\ge 1\\ m\equiv a \bmod{M}}}\sum_{\ell\ge 1}\frac{q^{\ell m}}{\ell}.
\end{equation}
Recall also that $X\ge 16$ and $N=\lfloor\sqrt{2\pi X}\rfloor$. As in Theorem \ref{th:main-est}, we put
\begin{equation}
q=e^{-\tau+\frac{2\pi i h}{k}},
\end{equation}
where $1\le h\le k\le N$ with $(h,k)=1$ and
\begin{equation}
\tau=X^{-1}+2\pi i Y
\end{equation}
with the restriction
\begin{equation}
|Y|\le \frac{1}{kN}.
\end{equation}

\subsection{Necessary bounds}

Now we are going to collect some bounds that will be frequently used in the sequel. First, the assumptions of $X$ and $N$ imply that
\begin{equation}
0.9\sqrt{2\pi X}\le N\le \sqrt{2\pi X}.
\end{equation}
Further, it follows from $N\le \sqrt{2\pi X}$ that
$$\frac{1}{X}\le \frac{2\pi}{N^2}\le \frac{2\pi}{kN}.$$
Therefore,
\begin{equation}\label{tau-upper-bound}
|\tau|\le \frac{2\sqrt{2}\pi}{kN}.
\end{equation}
Finally, we have
\begin{equation}\label{eq:1/k tau-sq}
\Re\left(\frac{1}{\tau}\right)\ge 0.07k^2.
\end{equation}
This is because
\begin{align*}
\Re\left(\frac{1}{k^2\tau}\right)&=\frac{X^{-1}}{k^2(X^{-2}+4\pi^2 Y^2)}\\
&\ge \frac{X^{-1}}{k^2(X^{-2}+4\pi^2 k^{-2}N^{-2})}\\
&\ge \frac{X^{-1}}{N^2 X^{-2}+4\pi^2N^{-2}}\\
&\ge \frac{X^{-1}}{(0.9\sqrt{2\pi X})^2 X^{-2}+4\pi^2(0.9\sqrt{2\pi X})^{-2}}\\[6pt]
&\ge 0.07.
\end{align*}

\subsection{Applying the inverse Mellin transform}

Given any positive integer $k$, we define
\begin{equation}\label{eq:K-def}
K:=k\frac{M}{(k,M)}.
\end{equation}
Notice that $M\mid K$. Write in \eqref{eq:Phi-def}
$$\ell=rk+\mu \qquad(1\le \mu\le k)$$
and
$$m=tK+\lambda\qquad(1\le \lambda\le K,\ \lambda\equiv a\bmod{M}).$$
Then
$$\Phi_{a,M}(q)=\sum_{\substack{1\le \lambda\le K\\ \lambda\equiv  a\bmod{M}}}\sum_{1\le \mu\le k}e^{\frac{2\pi i h \mu \lambda}{k}}\sum_{r,t\ge 0}\frac{1}{rk+\mu}e^{-(rk+\mu)(tK+\lambda)\tau}.$$
Applying the inverse Mellin transform further gives
\begin{align*}
\Phi_{a,M}(q)&=\sum_{\substack{1\le \lambda\le K\\ \lambda\equiv a\bmod{M}}}\sum_{1\le \mu\le k}e^{\frac{2\pi i h \mu \lambda}{k}}\sum_{r,t\ge 0}\frac{1}{2\pi i}\int_{(\tfrac{3}{2})}\frac{\Gamma(s)}{rk+\mu}\frac{ds}{(rk+\mu)^s(tK+\lambda)^s\tau^s}\\
&=\sum_{\substack{1\le \lambda\le K\\ \lambda\equiv a\bmod{M}}}\sum_{1\le \mu\le k}e^{\frac{2\pi i h \mu \lambda}{k}}\frac{1}{2\pi i}\int_{(\tfrac{3}{2})}\frac{\Gamma(s)}{\tau^s k^{s+1}K^s}\zrie{s,\frac{\lambda}{K}}\zrie{1+s,\frac{\mu}{k}}ds.
\end{align*}
Here the path of integration $(\alpha)$ is from $\alpha-i\infty$ to $\alpha+i\infty$.

Recall the functional equation of Hurwitz zeta function \cite[p.~261, Theorem 12.8]{Apo1976}:
\begin{align}\label{eq:Hur-fnc}
\zrie{s,\frac{\lambda}{\kappa}}&=2\Gamma(1-s)(2\pi \kappa)^{s-1}\Bigg(\sin\frac{\pi s}{2}\sum_{1\le \nu\le \kappa}\cos\frac{2\pi \lambda\nu}{\kappa}\;\zrie{1-s,\frac{\nu}{\kappa}}\nonumber\\
&\quad+\cos\frac{\pi s}{2}\sum_{1\le \nu\le \kappa}\sin\frac{2\pi \lambda\nu}{\kappa}\;\zrie{1-s,\frac{\nu}{\kappa}}\Bigg),
\end{align}
which is valid for $1\le \lambda\le \kappa$. Let us further put
\begin{equation}
z:=\frac{\tau k}{2\pi}.
\end{equation}
Now rewriting
\begin{align*}
	e^{\frac{2\pi i h \mu \lambda}{k}} = \cos \frac{2\pi h \mu \lambda}{k} + i \sin \frac{2\pi h \mu \lambda}{k}
\end{align*}
and applying \eqref{eq:Hur-fnc} to $\zrie{s,\frac{\lambda}{K}}$ with the use of Euler's reflection formula for the Gamma function \cite[p.~250, eq.~(3)]{Apo1976}:
\begin{align*}
	\Gamma(s)\Gamma(1-s) = \frac{\pi}{\sin (\pi s)} = \frac{\pi}{2\sin \frac{\pi s}{2} \cos \frac{\pi s}{2}},
\end{align*}
we derive that
\begin{align}\label{eq:Phi(q)-new-1}
	&\Phi_{a,M}(q)\notag\\
	&=\frac{1}{4\pi i k K} \sum_{\substack{1\le \lambda\le K\\ \lambda\equiv  a\bmod{M}}}\sum_{\substack{1\le\mu\le k\\1\le\nu\le K}} \cos\frac{2\pi h \mu \lambda}{k}\cos\frac{2\pi \nu \lambda}{K}\int_{(\tfrac{3}{2})}\frac{\zrie{1+s,\frac{\mu}{k}}\zrie{1-s,\frac{\nu}{K}}}{z^s \cos\frac{\pi s}{2}}ds\nonumber\\
	&\quad+\frac{1}{4\pi i k K} \sum_{\substack{1\le \lambda\le K\\ \lambda\equiv  a\bmod{M}}}\sum_{\substack{1\le\mu\le k\\1\le\nu\le K}} \cos\frac{2\pi h \mu \lambda}{k}\sin\frac{2\pi \nu \lambda}{K}\int_{(\tfrac{3}{2})}\frac{\zrie{1+s,\frac{\mu}{k}}\zrie{1-s,\frac{\nu}{K}}}{z^s \sin\frac{\pi s}{2}}ds\nonumber\\
	&\quad+\frac{1}{4\pi k K} \sum_{\substack{1\le \lambda\le K\\ \lambda\equiv  a\bmod{M}}}\sum_{\substack{1\le\mu\le k\\1\le\nu\le K}} \sin\frac{2\pi h \mu \lambda}{k}\sin\frac{2\pi \nu \lambda}{K}\int_{(\tfrac{3}{2})}\frac{\zrie{1+s,\frac{\mu}{k}}\zrie{1-s,\frac{\nu}{K}}}{z^s \sin\frac{\pi s}{2}}ds\nonumber\\
	&\quad+\frac{1}{4\pi k K} \sum_{\substack{1\le \lambda\le K\\ \lambda\equiv  a\bmod{M}}}\sum_{\substack{1\le\mu\le k\\1\le\nu\le K}} \sin\frac{2\pi h \mu \lambda}{k}\cos\frac{2\pi \nu \lambda}{K}\int_{(\tfrac{3}{2})}\frac{\zrie{1+s,\frac{\mu}{k}}\zrie{1-s,\frac{\nu}{K}}}{z^s \cos\frac{\pi s}{2}}ds.
\end{align}

\subsection{Changing the summation indices}

As in \cite[pp.~116--117]{Gro1958}, we shall then shift the path of integration back to $(-\tfrac{3}{2})$. However, before doing so, we first want to bijectively map the index $\lambda$ with $1\le \lambda\le K$ and $\lambda\equiv a \pmod{M}$ to a certain new index $\rho$. Such a change of summation indices is simply for computational convenience in the sequel. Also, the construction of the auxiliary function in \eqref{eq:Psi-def} is motivated from here.

To begin with, we define
\begin{align*}
	\sL&:=\big\{\lambda\,:\, \text{$1\le \lambda\le K$ and $\lambda\equiv a \bmod{M}$}\big\},\\
	\sR&:=\big\{\rho\,:\, \text{$1\le \rho\le k$ and $\rho\equiv -ha \bmod{M^*}$}\big\},
\end{align*}
where we put
$$M^*:=(k,M).$$
Keeping in mind that $M^*\mid k$ from the above and that $K=\frac{kM}{M^*}$ from \eqref{eq:K-def} so that $M\mid K$, we see that $\sL$ and $\sR$ are of equal cardinality $\frac{K}{M}=\frac{k}{M^*}$.

Now our object is to construct a one-to-one map from $\sL$ to $\sR$. For each $\lambda\in\sL$, we denote by $\rho=\rho(\lambda)$ the unique integer between $1$ and $k$ such that
\begin{equation}\label{eq:rho-def}
	\rho\equiv -h\lambda \pmod{k}.
\end{equation}
Recall that $\lambda\equiv a \pmod{M}$. Since $M^*\mid k$ and $M^*\mid M$, we have
\begin{equation}\label{eq:rho-M*}
	\rho\equiv -ha \pmod{M^*},
\end{equation}
thereby confirming that $\rho\in \sR$. To show that this map is injective, we choose $\lambda_1,\lambda_2\in \sL$. If $\rho(\lambda_1)=\rho(\lambda_2)$, then $-h \lambda_1\equiv -h \lambda_2\pmod{k}$ so that $\lambda_1\equiv \lambda_2\pmod{k}$ as $(h,k)=1$. Notice that $\lambda_1\equiv \lambda_2 \equiv a \pmod{M}$ implies that $\lambda_1$ is congruent to $\lambda_2$ modulo $\lcm(k,M)=\frac{kM}{(k,M)}=K$ by the Chinese remainder theorem. Now it follows from the definition of $\sL$ that $\lambda_1=\lambda_2$, as required.

For the inverse map, we first choose $h'$ so that
$$hh'\equiv -1 \pmod{k}.$$
This is always possible since $(h,k)=1$. Given any $\rho\in\sR$, consider the system
\begin{equation}
	\begin{cases}
		\lambda \equiv h'\rho & \pmod{k},\\
		\lambda\equiv a & \pmod{M},
	\end{cases}
\end{equation}
which is solvable whenever $h'\rho \equiv a \pmod{M^*}$. But this can be ensured by \eqref{eq:rho-M*} and the fact that $hh'\equiv -1 \pmod{M^*}$. By the Chinese remainder theorem, the solution is unique modulo $\lcm(k,M)=K$, and hence we get a unique $\lambda\in \sL$. To see that this $\lambda$ is indeed the preimage of $\rho$, we simply notice that
$$-h\lambda\equiv -h h'\rho\equiv \rho \pmod{k},$$
thereby satisfying \eqref{eq:rho-def}.
To explicitly express the residue class of $\lambda$ modulo $K$ for each $\rho\in \sR$, we find, using Euclid's algorithm, integers $\alpha$ and $\beta$ such that
\begin{equation}\label{eq:alpha&beta}
	\alpha k+\beta M=M^*.
\end{equation}
Therefore,
\begin{equation}\label{eq:lambda-K}
	\lambda\equiv a+\beta M \frac{h'\rho-a}{M^*}=\beta h'\frac{M}{M^*} \rho+\alpha a\frac{k}{M^*} \pmod{K}.
\end{equation}

From the above discussions, we may freely change summations over $\lambda\in \sL$ to summations over $\rho\in \sR$ with the same summand. Usually, the summand contains both $\lambda$ and $\rho$ but it should be kept in mind that they are \emph{dependent} to one another by the above bijective map. In other words,
\begin{align}\label{eq:sum=sum}
	\sum_{\substack{1\le \lambda\le K\\ \lambda\equiv  a\bmod{M}}} \textsf{(summand)} = \sum_{\substack{1\le \rho\le k\\ \rho\equiv b\bmod{M^*}}} \textsf{(summand)},
\end{align}
where we recall from Theorem \ref{th:main-est} that $b$ is defined as the unique integer between $1$ and $M^*$ such that
\begin{equation}
	b\equiv -ha \pmod{M^*}.
\end{equation}

\begin{remark}\label{rmk:3-cases}
	Notice that in the case where $a\equiv 0 \pmod{M}$, we may take $\lambda=K$. It is then true that $b\equiv -h\cdot 0\equiv 0 \pmod{M^*}$ and that the corresponding $\rho$ of $\lambda=K$ is such that $\rho\equiv -hK\equiv 0 \pmod{k}$ so that $\rho=k$. However, the opposite is not necessarily valid. In other words, when $b\equiv 0 \pmod{M^*}$, we may take $\rho=k$. But it is not necessarily true that $a\equiv 0 \pmod{M}$. Also, the corresponding $\lambda$ of $\rho=k$ is such that
	\begin{align*}
		\begin{cases}
			\lambda \equiv 0 & \pmod{k},\\
			\lambda\equiv a & \pmod{M}.
		\end{cases}
	\end{align*}
	So when $a\not\equiv 0 \pmod{M}$, we find that $\lambda\not\equiv 0 \pmod{K}$ and hence that $\lambda\ne K$, while when $a\equiv 0 \pmod{M}$, we have $\lambda\equiv 0 \pmod{K}$, yielding that $\lambda= K$. To summarize the above discussion, we have three cases.
	\begin{itemize}[leftmargin=*,align=left]
		\renewcommand{\labelitemi}{\scriptsize$\blacktriangleright$}
		
		\item
		\textit{Case 1.}~$a\not\equiv 0 \pmod{M}$ and $b\not\equiv 0 \pmod{M^*}$: In this case, $\lambda$ cannot take the value $K$ and $\rho$ cannot take the value $k$.
		
		\item
		\textit{Case 2.}~$a\not\equiv 0 \pmod{M}$ and $b\equiv 0 \pmod{M^*}$: In this case, $\lambda$ cannot take the value $K$ but $\rho$ may take the value $k$. In particular, the corresponding $\lambda$ of $\rho=k$ is a multiple of $k$.
		
		\item
		\textit{Case 3.}~$a\equiv 0 \pmod{M}$ and $b\equiv 0 \pmod{M^*}$: In this case, $\lambda$ may take the value $K$ and $\rho$ may take the value $k$. In particular, the corresponding $\lambda$ of $\rho=k$ is $K$ and vice versa.
	\end{itemize}
\end{remark}

\subsection{Shifting the path of integration}

In \eqref{eq:Phi(q)-new-1}, replacing $s$ by $-s$, reversing the direction of the integration path and shifting the path back to $(\tfrac{3}{2})$, one has, with $h\lambda$ replaced by $-\rho$ in light of \eqref{eq:rho-def},
\begin{align}\label{eq:Phi(q)-new-2}
	&\Phi_{a,M}(q)\notag\\
	&=\frac{1}{4\pi i k K} \sum_{\substack{1\le \lambda\le K\\ \lambda\equiv  a\bmod{M}}}\sum_{\substack{1\le\mu\le k\\1\le\nu\le K}} \cos\frac{2\pi \mu\rho}{k}\cos\frac{2\pi \nu \lambda}{K}\int_{(\tfrac{3}{2})}\frac{\zrie{1-s,\frac{\mu}{k}}\zrie{1+s,\frac{\nu}{K}}}{z^{-s} \cos\frac{\pi s}{2}}ds\nonumber\\
	&\quad-\frac{1}{4\pi i k K} \sum_{\substack{1\le \lambda\le K\\ \lambda\equiv  a\bmod{M}}}\sum_{\substack{1\le\mu\le k\\1\le\nu\le K}} \cos\frac{2\pi \mu\rho}{k}\sin\frac{2\pi \nu \lambda}{K}\int_{(\tfrac{3}{2})}\frac{\zrie{1-s,\frac{\mu}{k}}\zrie{1+s,\frac{\nu}{K}}}{z^{-s} \sin\frac{\pi s}{2}}ds\nonumber\\
	&\quad+\frac{1}{4\pi k K} \sum_{\substack{1\le \lambda\le K\\ \lambda\equiv  a\bmod{M}}}\sum_{\substack{1\le\mu\le k\\1\le\nu\le K}} \sin\frac{2\pi \mu\rho}{k}\sin\frac{2\pi \nu \lambda}{K}\int_{(\tfrac{3}{2})}\frac{\zrie{1-s,\frac{\mu}{k}}\zrie{1+s,\frac{\nu}{K}}}{z^{-s} \sin\frac{\pi s}{2}}ds\nonumber\\
	&\quad-\frac{1}{4\pi k K} \sum_{\substack{1\le \lambda\le K\\ \lambda\equiv  a\bmod{M}}}\sum_{\substack{1\le\mu\le k\\1\le\nu\le K}} \sin\frac{2\pi \mu\rho}{k}\cos\frac{2\pi \nu \lambda}{K}\int_{(\tfrac{3}{2})}\frac{\zrie{1-s,\frac{\mu}{k}}\zrie{1+s,\frac{\nu}{K}}}{z^{-s} \cos\frac{\pi s}{2}}ds\nonumber\\
	&\quad-2\pi i(R_1+R_2+R_3+R_4)\nonumber\\
	&=: \Upsilon_1+\Upsilon_2+\Upsilon_3+\Upsilon_4-2\pi i(R_1+R_2+R_3+R_4),
\end{align}
where the $R_j$'s come from the sum of residues of the corresponding integrand inside the strip $-\frac{3}{2}<\Re(s)<\frac{3}{2}$.

\begin{remark}
	Here we essentially shift the path of integration from $(\tfrac{3}{2})$ to $(-\tfrac{3}{2})$. As pointed out by one of the referees, it would be interesting to investigate what happens if one pulls back the contour further to $(-\tfrac{5}{2})$, etc.
\end{remark}

\subsection{Proof of Theorem \ref{th:main-est}}

To prove Theorem \ref{th:main-est}, it remains to study the integrals $\Upsilon_j$ and the residues $R_j$, and we will do so in the next two sections. In particular, the main term in \eqref{eq:th-main} comes from $R_1$ and $R_4$; see \eqref{eq:R11} and \eqref{eq:R41}. For the error term, we combine the estimations in \eqref{eq:bound-integral}, \eqref{eq:R12}, \eqref{eq:R13}, \eqref{eq:R14}, \eqref{eq:R21}, \eqref{eq:R22}, \eqref{eq:R23}, \eqref{eq:R3} and \eqref{eq:R42}.

\section{Integrals}\label{sec:integrals}

\subsection{An auxiliary function}

Let us define an auxiliary function
\begin{align}\label{eq:Psi-def}
	\Psi_{a,M}(q):=\log\left(\prod_{\substack{m\ge 1\\m\equiv -ha\bmod{M^*}}}\frac{1}{1-e^{\frac{2\pi i \alpha a}{M}}q^{m}}\right),
\end{align}
where $\alpha$ is as in \eqref{eq:alpha&beta}. We further write
$$m=rk+\rho\qquad(1\le \rho\le k,\ \rho\equiv -ha\bmod{M^*}),$$
and put
\begin{equation}\label{eq:q^*}
	q^*:=\exp\left(\frac{2\pi i\beta h'}{k}-\frac{2\pi}{K z}\right),
\end{equation}
where $\beta$ is again as in \eqref{eq:alpha&beta}. Then
$$\Psi_{a,M}(q^*)=-\sum_{\substack{1\le \rho\le k\\ \rho\equiv -ha\bmod{M^*}}}\sum_{r\ge 0}\log\Bigg(1-\exp\bigg(\frac{2\pi i\beta h'}{k}\rho-\frac{2\pi}{K z}(rk+\rho)+\frac{2\pi i \alpha a}{M}\bigg)\Bigg).$$
It follows from \eqref{eq:lambda-K} that
$$\exp\left(\frac{2\pi i \lambda}{K}\right)=\exp\left(\frac{2\pi i \beta h'M}{K M^*} \rho+\frac{2\pi i \alpha a k}{K M^*}\right)=\exp\left(\frac{2\pi i \beta h'}{k} \rho+\frac{2\pi i \alpha a}{M}\right).$$
Hence,
\begin{align*}
	&\Psi_{a,M}(q^*)\\
	&=-\sum_{\substack{1\le \rho\le k\\ \rho\equiv -ha\bmod{M^*}}}\sum_{r\ge 0}\log\Bigg(1-\exp\bigg({-\frac{2\pi}{K z}}(rk+\rho)+\frac{2\pi i \lambda}{K}\bigg)\Bigg)\\
	&=\sum_{\substack{1\le \rho\le k\\ \rho\equiv -ha\bmod{M^*}}}\sum_{1\le\nu\le K}\sum_{r,t\ge 0}\frac{1}{tK+\nu}\exp\Bigg((tK+\nu)\bigg({-\frac{2\pi}{K z}}(rk+\rho)+\frac{2\pi i \lambda}{K}\bigg)\Bigg)\\
	&=\sum_{\substack{1\le \rho\le k\\ \rho\equiv -ha\bmod{M^*}}}\sum_{1\le \nu\le K}e^{\frac{2\pi i \nu \lambda}{K}}\sum_{r,t\ge 0}\frac{1}{tK+\nu}e^{-(rk+\rho)(tK+\nu)\frac{2\pi}{Kz}}.
\end{align*}

If we substitute $\rho$ back to $\lambda$ by \eqref{eq:lambda-K} and apply the inverse Mellin transform and the functional equation of Hurwitz zeta function to $\Psi_{a,M}(q^*)$, then
\begin{align}\label{eq:Psi(q)-new-1}
	&\Psi_{a,M}(q^*)\notag\\
	&=\frac{1}{4\pi i k K} \sum_{\substack{1\le \lambda\le K\\ \lambda\equiv  a\bmod{M}}}\sum_{\substack{1\le\mu\le k\\1\le\nu\le K}} \cos\frac{2\pi \mu\rho}{k}\cos\frac{2\pi \nu \lambda}{K}\int_{(\tfrac{3}{2})}\frac{\zrie{1-s,\frac{\mu}{k}}\zrie{1+s,\frac{\nu}{K}}}{z^{-s} \cos\frac{\pi s}{2}}ds\nonumber\\
	&\quad+\frac{1}{4\pi i k K} \sum_{\substack{1\le \lambda\le K\\ \lambda\equiv  a\bmod{M}}}\sum_{\substack{1\le\mu\le k\\1\le\nu\le K}} \sin\frac{2\pi \mu\rho}{k}\cos\frac{2\pi \nu \lambda}{K}\int_{(\tfrac{3}{2})}\frac{\zrie{1-s,\frac{\mu}{k}}\zrie{1+s,\frac{\nu}{K}}}{z^{-s} \sin\frac{\pi s}{2}}ds\nonumber\\
	&\quad+\frac{1}{4\pi k K} \sum_{\substack{1\le \lambda\le K\\ \lambda\equiv  a\bmod{M}}}\sum_{\substack{1\le\mu\le k\\1\le\nu\le K}} \sin\frac{2\pi \mu\rho}{k}\sin\frac{2\pi \nu \lambda}{K}\int_{(\tfrac{3}{2})}\frac{\zrie{1-s,\frac{\mu}{k}}\zrie{1+s,\frac{\nu}{K}}}{z^{-s} \sin\frac{\pi s}{2}}ds\nonumber\\
	&\quad+\frac{1}{4\pi k K} \sum_{\substack{1\le \lambda\le K\\ \lambda\equiv  a\bmod{M}}}\sum_{\substack{1\le\mu\le k\\1\le\nu\le K}} \cos\frac{2\pi \mu\rho}{k}\sin\frac{2\pi \nu \lambda}{K}\int_{(\tfrac{3}{2})}\frac{\zrie{1-s,\frac{\mu}{k}}\zrie{1+s,\frac{\nu}{K}}}{z^{-s} \cos\frac{\pi s}{2}}ds\nonumber\\
	&=: J_1+J_2+J_3+J_4.
\end{align}
Notice that
\begin{align*}
	\Upsilon_1 = J_1\qquad\text{and}\qquad\Upsilon_3 = J_3.
\end{align*}
Further,
\begin{equation}
	2(J_1+J_3)=\Psi_{a,M}(q^*)+\Psi_{M-a,M}(q^*).
\end{equation}

\subsection{Estimations concerning Hurwitz zeta functions}

Recall that for $\Re(s)>1$ and $0<\alpha\le 1$,
$$\zrie{1-s,\alpha}=\frac{2\Gamma(s)}{(2\pi)^{s}}\sum_{n=1}^{%
	\infty}\frac{1}{n^{s}}\cos\left(\frac{\pi s}{2}-2n\pi \alpha\right);$$
this can be found in, for instance, \cite[p.~608, eq.~(25.11.9)]{Apo2010}. Thus, for $0<\alpha\le 1$, we have a uniform bound
\begin{equation}
	|\zrie{-0.5+it,\alpha}|\le \frac{2\Gamma(\tfrac{3}{2})\zeta(\tfrac{3}{2})\cosh(\tfrac{\pi|t|}{2})}{(2\pi)^{\frac{3}{2}}}.
\end{equation}
It also follows from \cite[Theorem 12.23]{Apo1976} with some simple calculations that, uniformly for $|t|\ge 3$ and $0<\alpha\le 1$, 
\begin{equation}
	|\zrie{-0.5+it,\alpha}|\le 11 |t|^{\frac{3}{2}}.
\end{equation}
Finally, we have, for $0<\alpha\le 1$,
\begin{equation}
	|\zrie{2.5+it,\alpha}|\le \alpha^{-\frac{5}{2}}+\zeta(\tfrac{5}{2}).
\end{equation}

\begin{lemma}\label{le:I+I-}
	Let $z$ be a complex number with $\Re(z)>0$ and let $0<\alpha,\beta\le 1$. Define integrals
	\begin{align}
		\mathcal{I}_+(z):=\int_{(\tfrac{3}{2})}z^s\zrie{1+s,\alpha}\zrie{1-s,\beta}\Bigg(\frac{1}{\cos\frac{\pi s}{2}}+\frac{1}{i \sin\frac{\pi s}{2}}\Bigg)ds
	\end{align}
	and
	\begin{align}
		\mathcal{I}_-(z):=\int_{(\tfrac{3}{2})}z^s\zrie{1+s,\alpha}\zrie{1-s,\beta}\Bigg(\frac{1}{\cos\frac{\pi s}{2}}-\frac{1}{i \sin\frac{\pi s}{2}}\Bigg)ds.
	\end{align}
	Then if $\Im(z)\le 0$, we have
	\begin{equation}
		|\mathcal{I}_+(z)|\le 7.23 |z|^{\frac{3}{2}}\big(\alpha^{-\frac{5}{2}}+\zeta(\tfrac{5}{2})\big),
	\end{equation}
	while if $\Im(z)\ge 0$, we have
	\begin{equation}
		|\mathcal{I}_-(z)|\le 7.23 |z|^{\frac{3}{2}}\big(\alpha^{-\frac{5}{2}}+\zeta(\tfrac{5}{2})\big).
	\end{equation}
\end{lemma}

\begin{proof}
	Writing $s=\frac{3}{2}+it$ since the path of integration is the vertical line $\Re(s)=\frac{3}{2}$, we have
	$$|z^s|=|z|^{\frac{3}{2}}e^{-\Arg(z)t}.$$
	Also,
	$$\left|\frac{1}{\cos\frac{\pi s}{2}}+\frac{1}{i \sin\frac{\pi s}{2}}\right|=\frac{2e^{-\frac{\pi}{2}t}}{|\sin(\pi s)|}.$$
	Hence, for $z$ with $\Im(z)\le 0$ (recall that $\Re(z)>0$ so that $-\frac{\pi}{2}<\Arg(z)\le 0$), we have
	$$|z^s|\left|\frac{1}{\cos\frac{\pi s}{2}}+\frac{1}{i \sin\frac{\pi s}{2}}\right|\le 2|z|^{\frac{3}{2}}\frac{e^{\frac{\pi}{2}|t|}}{|\sin(\pi s)|}.$$
	It follows that
	\begin{align*}
		|\mathcal{I}_+(z)|&\le 2|z|^{\frac{3}{2}}\big((\alpha^{-\frac{5}{2}}+\zeta(\tfrac{5}{2})\big)\int_{-\infty}^\infty \left|\zrie{-0.5-it,\beta}\right|\frac{e^{\frac{\pi}{2}|t|}}{|\sin(\pi (1.5+it))|}dt\\
		&\le 7.23 |z|^{\frac{3}{2}}\big(\alpha^{-\frac{5}{2}}+\zeta(\tfrac{5}{2})\big).
	\end{align*}
	Similar arguments also apply to $\mathcal{I}_-(z)$ if $\Im(z)\ge 0$.
\end{proof}

\subsection{Bounding the integrals}

Recall that
$$z=\frac{\tau k}{2\pi}.$$
For $\Upsilon_2$ and $\Upsilon_4$, we define, with $j=2$ or $4$,
\begin{equation}
	\Upsilon_j\pm J_j:=\begin{cases}
		\Upsilon_j+ J_j & \text{if $\Im(z)\ge 0$},\\
		\Upsilon_j- J_j & \text{if $\Im(z)<0$}.
	\end{cases}
\end{equation}
It follows from Lemma \ref{le:I+I-} that
\begin{align}
	|\Upsilon_j\pm J_j|&\le \frac{1}{4\pi k K} \frac{k K}{M} \sum_{1\le\nu\le K} 7.23\; |z|^{\frac{3}{2}}\Bigg(\left(\frac{K}{\nu}\right)^{\frac{5}{2}}+\zeta(\tfrac{5}{2})\Bigg)\notag\\
	&\le \frac{1}{4\pi M}\cdot 7.23\; |z|^{\frac{3}{2}} \cdot 2\zeta(\tfrac{5}{2})K^{\frac{5}{2}}\notag\\
	&\le \frac{7.23\;\zeta(\tfrac{5}{2})}{2\pi M}\left|\frac{\tau k}{2\pi}\right|^{\frac{3}{2}}\left(k\frac{M}{(k,M)}\right)^{\frac{5}{2}}\notag\\
	\text{\tiny (by \eqref{tau-upper-bound})}&\le \frac{7.23\;\zeta(\tfrac{5}{2})}{2\pi M}\left(\frac{\sqrt{2}}{N}\right)^{\frac{3}{2}} \left(\frac{M}{(k,M)}\right)^{\frac{5}{2}} N^{\frac{5}{2}}\notag\\
	&\le \frac{7.23\;\zeta(\tfrac{5}{2})\;2^{\frac{3}{4}}}{2\pi M}\left(\frac{M}{(k,M)}\right)^{\frac{5}{2}}\sqrt{2\pi X}\notag\\
	&\le 6.51\frac{M^{\frac{3}{2}}}{(k,M)^{\frac{5}{2}}} X^{\frac{1}{2}}\notag\\
	&\ll X^{\frac{1}{2}}.
\end{align}

On the other hand, we recall from \eqref{eq:q^*} that
$$q^*=\exp\left(\frac{2\pi i\beta h'}{k}-\frac{2\pi}{K z}\right).$$
Hence,
\begin{align*}
	|q^*|=\exp\left(\Re\left(-\frac{2\pi}{Kz}\right)\right)=\exp\left(-4\pi^2\frac{(k,M)}{M}\Re\left(\frac{1}{k^2\tau}\right)\right).
\end{align*}
It follows from \eqref{eq:1/k tau-sq} that
\begin{equation}\label{eq:q^*-bound}
	|q^*|\le \exp\left(-4\pi^2\frac{(k,M)}{M}\cdot 0.07\right)\ll 1.
\end{equation}
We further deduce from some simple partition-theoretic arguments that, for any $a$ among $1,2,\ldots,M$,
\begin{align*}
	e^{|\Re(\Psi_{a,M}(q^*))|}&\le \prod_{\substack{m\ge 1\\m\equiv -ha\bmod{M^*}}}\frac{1}{1-|q^*|^{m}}\le \frac{1}{(|q^*|;|q^*|)_\infty}\\
	&=\exp\left(-\sum_{\ell\ge 1}\log(1-|q^*|^{\ell})\right)=\exp\left(\sum_{\ell\ge 1}\sum_{m\ge 1}\frac{|q^*|^{\ell m}}{m}\right)\\
	&\le \exp\left(\sum_{n\ge 1} n|q^*|^n\right)=\exp\left(\frac{|q^*|}{(1-|q^*|)^2}\right).
\end{align*}
In consequence,
\begin{align*}
	|\Re(\Psi_{a,M}(q^*))|\le \frac{e^{-0.28 \pi^2\frac{(k,M)}{M}}}{\left(1-e^{-0.28 \pi^2\frac{(k,M)}{M}}\right)^2}\ll 1.
\end{align*}

Finally, we bound
\begin{align*}
	&|\Re(\Upsilon_1+\Upsilon_2+\Upsilon_3+\Upsilon_4)|\\
	&\qquad\le |\Re(\Upsilon_1+\Upsilon_3)|+|\Re(\Upsilon_2+\Upsilon_4)|\\
	&\qquad\le |\Re(J_1+J_3)|+|\Re(J_2+J_4)|+|\Upsilon_2\pm J_2|+|\Upsilon_4\pm J_4|\\
	&\qquad\le |\Re(\Psi_{a,M}(q^*))|+2|\Re(J_1+J_3)|+|\Upsilon_2\pm J_2|+|\Upsilon_4\pm J_4|\\
	&\qquad\le 2|\Re(\Psi_{a,M}(q^*))|+|\Re(\Psi_{M-a,M}(q^*))|+|\Upsilon_2\pm J_2|+|\Upsilon_4\pm J_4|.
\end{align*}
Here for the marginal case where $a=M$, we simply notice the fact that $\Psi_{0,M}(q^*)=\Psi_{M,M}(q^*)$. It turns out that
\begin{align*}
	|\Re(\Upsilon_1+\Upsilon_2+\Upsilon_3+\Upsilon_4)|\le 3\cdot \frac{e^{-0.28 \pi^2\frac{(k,M)}{M}}}{\left(1-e^{-0.28 \pi^2\frac{(k,M)}{M}}\right)^2}+2\cdot 6.51\frac{M^{\frac{3}{2}}}{(k,M)^{\frac{5}{2}}} X^{\frac{1}{2}}.
\end{align*}
The above evaluation is summarized as follows.

\begin{theorem}
	We have
	\begin{align}\label{eq:bound-integral}
		|\Re(\Upsilon_1+\Upsilon_2+\Upsilon_3+\Upsilon_4)|\le \frac{3e^{-0.28 \pi^2\frac{(k,M)}{M}}}{\left(1-e^{-0.28 \pi^2\frac{(k,M)}{M}}\right)^2}+13.02\frac{M^{\frac{3}{2}}}{(k,M)^{\frac{5}{2}}} X^{\frac{1}{2}}
		\ll X^{\frac{1}{2}}.
	\end{align}
\end{theorem}

\section{Residues}\label{sec:res}

\subsection{Lemmas}\label{sec:res-lemmas}

We first require some finite summation formulas of Hurwitz zeta functions, which follow from the first two aligned formulas on page 587 of \cite{Bla2015}.

\begin{lemma}
	For any $\theta=1,2,\ldots,k$,
	\begin{align}\label{eq:sum-cos-zeta-0}
		\sum_{1\le\alpha\le k}\cos\frac{2\pi \alpha\theta}{k}\zrie{0,\frac{\alpha}{k}}=-\frac{1}{2}
	\end{align}
	and
	\begin{align}\label{eq:sum-cos-zeta-2}
		\sum_{1\le\alpha\le k}\cos\frac{2\pi \alpha\theta}{k}\zrie{2,\frac{\alpha}{k}}=\frac{\pi^2}{6}(6\theta^2-6k\theta+k^2).
	\end{align}
	For any $\theta=1,2,\ldots,k-1$,
	\begin{align}\label{eq:sum-sin-zeta-0}
		\sum_{1\le\alpha\le k}\sin\frac{2\pi \alpha\theta}{k}\zrie{0,\frac{\alpha}{k}}=\frac{1}{2\pi}\Bigg(\dg{1-\frac{\theta}{k}}-\dg{\frac{\theta}{k}}\Bigg)=\frac{1}{2}\cot\frac{\pi \theta}{k}
	\end{align}
	and
	\begin{align}\label{eq:sum-sin-zeta-2}
		\sum_{1\le\alpha\le k}\sin\frac{2\pi \alpha\theta}{k}\zrie{2,\frac{\alpha}{k}}&=2\pi k^2\Bigg(\zeta'\left(-1,\frac{\theta}{k}\right)-\zeta'\left(-1,1-\frac{\theta}{k}\right)\Bigg).
	\end{align}
\end{lemma}

We also need three finite summation formulas of the digamma functions due to Gau\ss{} (cf.~\cite[p.~19, eq.~(49)]{SC2001} or \cite[eqs.~(77) and (B.6)]{Bla2015}).

\begin{lemma}
	For any $\theta=1,2,\ldots,k-1$,
	\begin{align}\label{eq:sum-cos-gamma}
		\sum_{1\le\alpha\le k}\cos\frac{2\pi \alpha\theta}{k}\dg{\frac{\alpha}{k}}=k \log\left(2\sin\frac{\pi \theta}{k}\right)
	\end{align}
	and
	\begin{align}\label{eq:sum-sin-gamma}
		\sum_{1\le\alpha\le k}\sin\frac{2\pi \alpha\theta}{k}\dg{\frac{\alpha}{k}}=\frac{\pi}{2}(2\theta-k).
	\end{align}
	Further,
	\begin{align}\label{eq:sum-gamma}
		\sum_{1\le\alpha\le k}\dg{\frac{\alpha}{k}}=-k(\gamma+\log k).
	\end{align}
\end{lemma}

Let $\gamma_n(\alpha)$ be the $n$-th generalized Stieltjes constant defined by
\begin{align*}
	\zeta(s,\alpha)=-\frac{1}{1-s}+\sum_{n\ge 0}\frac{\gamma_n(\alpha)}{n!}(1-s)^n.
\end{align*}
The following evaluation in \cite[p.~562, Theorem 2]{Bla2015} is necessary.

\begin{lemma}
	For any $\theta=1,2,\ldots,k-1$,
	\begin{align}\label{eq:sum-Sti}
		\sum_{1\le \alpha\le k}\sin\frac{2\pi \alpha\theta}{k}\gamma_1\!\left(\frac{\alpha}{k}\right)&=\frac{\pi}{2}\big(\gamma+\log(2\pi k)\big)(2\theta-k)-\frac{\pi \log(2\pi)}{2}k\notag\\
		&\quad+\frac{\pi k}{2}\log\left(2\sin\frac{\pi \theta}{k}\right)+\pi k \log \Gamma\left(\frac{\theta}{k}\right).
	\end{align}
\end{lemma}

It is notable that in $R_j$ with $1\le j\le 4$, the residues are essentially from
$$\frac{\zrie{1-s,\frac{\mu}{k}}\zrie{1+s,\frac{\nu}{K}}}{z^{-s} \operatorname{trig}\frac{\pi s}{2}},$$
where the $\operatorname{trig}$ function is $\cos$ or $\sin$. Now we compute that
\begin{align*}
	\cR_c:\!&=\sum_{|\Re(s)|< \frac{3}{2}}\Res_s\frac{\zrie{1-s,\frac{\mu}{k}}\zrie{1+s,\frac{\nu}{K}}}{z^{-s} \cos\frac{\pi s}{2}}= \underset{s=0}{\Res}\,(*)+\underset{s=-1}{\Res}\,(*)+\underset{s=1}{\Res}\,(*)\\
	&=-\log z-\dg{\frac{\mu}{k}}+\dg{\frac{\nu}{K}}+\frac{2\zrie{2,\frac{\mu}{k}}\zrie{0,\frac{\nu}{K}}}{\pi z}-\frac{2z\zrie{0,\frac{\mu}{k}}\zrie{2,\frac{\nu}{K}}}{\pi},
\end{align*}
and that
\begin{align*}
	\cR_s:\!&=\sum_{|\Re(s)|< \frac{3}{2}}\Res_s\frac{\zrie{1-s,\frac{\mu}{k}}\zrie{1+s,\frac{\nu}{K}}}{z^{-s} \sin\frac{\pi s}{2}}=\underset{s=0}{\Res}\,(*)\\
	&=-\frac{\pi}{12}-\frac{(\log z)^2}{\pi}-\frac{2\log z}{\pi}\dg{\frac{\mu}{k}}+\frac{2\log z}{\pi}\dg{\frac{\nu}{K}}\\[6pt]
	&\quad+\frac{2}{\pi}\dg{\frac{\mu}{k}}\dg{\frac{\nu}{K}}+\frac{2}{\pi}\gamma_1\!\left(\frac{\mu}{k}\right)+\frac{2}{\pi}\gamma_1\!\left(\frac{\nu}{K}\right).
\end{align*}

Finally, let us recall the three cases distinguished in Remark \ref{rmk:3-cases}:
\begin{itemize}[leftmargin=*,align=left]
	\renewcommand{\labelitemi}{\scriptsize$\blacktriangleright$}
	
	\item
	\textit{Case 1.}~$a\not\equiv 0 \pmod{M}$ and $b\not\equiv 0 \pmod{M^*}$: In this case, $\lambda$ cannot take the value $K$ and $\rho$ cannot take the value $k$.
	
	\item
	\textit{Case 2.}~$a\not\equiv 0 \pmod{M}$ and $b\equiv 0 \pmod{M^*}$: In this case, $\lambda$ cannot take the value $K$ but $\rho$ may take the value $k$. In particular, the corresponding $\lambda$ of $\rho=k$, denoted by $\lambda_k$, is a multiple of $k$. Thus, recalling that $K=k\frac{M}{M^*}$, we have $k\le \lambda_k\le k(\frac{M}{M^*}-1)$, that is,
	\begin{align}
		\frac{M^*}{M}\le \frac{\lambda_k}{K}\le 1-\frac{M^*}{M}.
	\end{align}
	
	\item
	\textit{Case 3.}~$a\equiv 0 \pmod{M}$ and $b\equiv 0 \pmod{M^*}$: In this case, $\lambda$ may take the value $K$ and $\rho$ may take the value $k$. In particular, the corresponding $\lambda$ of $\rho=k$ is $K$ and vice versa.
\end{itemize}
The latter two cases may lead to some marginal contributions to $R_1$ and $R_2$ when processing the evaluations of
\begin{align*}
	\sum_{1\le\mu\le k}\cos\frac{2\pi \mu\rho}{k}\quad\text{and}\quad\sum_{1\le\nu\le K}\cos\frac{2\pi \nu \lambda}{K}.
\end{align*}

\subsection{Evaluation of $R_1$}

We have
\begin{align*}
	R_1=\frac{1}{4\pi i k K} \sum_{\substack{1\le \lambda\le K\\ \lambda\equiv  a\bmod{M}}}\sum_{\substack{1\le\mu\le k\\1\le\nu\le K}} \cos\frac{2\pi \mu\rho}{k}\cos\frac{2\pi \nu \lambda}{K}\cdot \cR_c.
\end{align*}
Now write
$$R_1=R_{11}+R_{12}+R_{13}+R_{14},$$
where
\begin{align*}
	R_{11}&:=\frac{1}{4\pi i k K} \sum_{\substack{1\le \lambda\le K\\ \lambda\equiv  a\bmod{M}}}\sum_{\substack{1\le\mu\le k\\1\le\nu\le K}} \cos\frac{2\pi \mu\rho}{k}\cos\frac{2\pi \nu \lambda}{K}\cdot\frac{2\zrie{2,\frac{\mu}{k}}\zrie{0,\frac{\nu}{K}}}{\pi z},\\
	R_{12}&:=\frac{1}{4\pi i k K} \sum_{\substack{1\le \lambda\le K\\ \lambda\equiv  a\bmod{M}}}\sum_{\substack{1\le\mu\le k\\1\le\nu\le K}} \cos\frac{2\pi \mu\rho}{k}\cos\frac{2\pi \nu \lambda}{K}\cdot\left(-\frac{2z\zrie{0,\frac{\mu}{k}}\zrie{2,\frac{\nu}{K}}}{\pi}\right),\\
	R_{13}&:=\frac{1}{4\pi i k K} \sum_{\substack{1\le \lambda\le K\\ \lambda\equiv  a\bmod{M}}}\sum_{\substack{1\le\mu\le k\\1\le\nu\le K}} \cos\frac{2\pi \mu\rho}{k}\cos\frac{2\pi \nu \lambda}{K}\cdot(-\log z),\\
	R_{14}&:=\frac{1}{4\pi i k K} \sum_{\substack{1\le \lambda\le K\\ \lambda\equiv  a\bmod{M}}}\sum_{\substack{1\le\mu\le k\\1\le\nu\le K}} \cos\frac{2\pi \mu\rho}{k}\cos\frac{2\pi \nu \lambda}{K}\cdot\left(-\dg{\frac{\mu}{k}}+\dg{\frac{\nu}{K}}\right).
\end{align*}

\begin{theorem}
	We have
	\begin{equation}\label{eq:R11}
		{-2}\pi iR_{11}=\frac{1}{\tau}\frac{\pi^2}{6 k^2 M}\big(6b^2-6b(k,M)+(k,M)^2\big).
	\end{equation}
	Also,
	\begin{align}
		|\Re({-2}\pi i R_{12})|&\le \tfrac{1}{24}MX^{-1}\ll X^{-1},\label{eq:R12}\\
		|\Re({-2}\pi i R_{13})|&\le \tfrac{1}{2}\log X+0.92\ll \log X,\label{eq:R13}\\
		|\Re({-2}\pi i R_{14})|&\le \tfrac{1}{2}\tfrac{M}{(k,M)}\ll 1.\label{eq:R14}
	\end{align}
\end{theorem}

\begin{proof}
	We begin with the evaluation of $R_{11}$ and derive that
	\begin{align}\label{eq:R11-0}
		R_{11}&=\frac{1}{z}\frac{1}{2i \pi^2 k K}\sum_{\substack{1\le \lambda\le K\\ \lambda\equiv  a\bmod{M}}}\sum_{1\le \mu\le k}\cos\frac{2\pi \mu\rho}{k}\zrie{2,\frac{\mu}{k}}\sum_{1\le \nu\le K}\cos\frac{2\pi \nu \lambda}{K}\zrie{0,\frac{\nu}{K}}\notag\\
		\text{\tiny (by $\substack{\eqref{eq:sum-cos-zeta-0}\\ \eqref{eq:sum-cos-zeta-2}}$)}&=\frac{1}{z}\frac{1}{2i \pi^2 k K}\sum_{\substack{1\le \rho\le k\\ \rho\equiv b\bmod{M^*}}}\frac{\pi^2}{6}(6\rho^2-6k\rho+k^2)\cdot \left(-\frac{1}{2}\right)\notag\\
		&=\frac{1}{z}\frac{1}{2i \pi^2 k K}\cdot \frac{\pi^2}{6}\frac{k}{M^*}\big(6b^2-6bM^*+(M^*)^2\big)\cdot \left(-\frac{1}{2}\right)\notag\\
		&=-\frac{2\pi}{\tau k} \frac{1}{24i k M}\big(6b^2-6bM^*+(M^*)^2\big),
	\end{align}
	thereby yielding \eqref{eq:R11}.
	
	Similarly, for $R_{12}$, it is computed that
	\begin{align}
		R_{12}&=-z\frac{1}{2i \pi^2 k K}\sum_{\substack{1\le \lambda\le K\\ \lambda\equiv  a\bmod{M}}}\sum_{1\le \mu\le k}\cos\frac{2\pi \mu\rho}{k}\zrie{0,\frac{\mu}{k}}\sum_{1\le \nu\le K}\cos\frac{2\pi \nu \lambda}{K}\zrie{2,\frac{\nu}{K}}\notag\\
		\text{\tiny (by $\substack{\eqref{eq:sum-cos-zeta-0}\\ \eqref{eq:sum-cos-zeta-2}}$)}&=-z\frac{1}{2i \pi^2 k K}\sum_{\substack{1\le \lambda\le K\\ \lambda\equiv  a\bmod{M}}}\left(-\frac{1}{2}\right)\cdot \frac{\pi^2}{6}(6\lambda^2-6k\lambda+k^2)\notag\\
		&=-z\frac{1}{2i \pi^2 k K}\cdot \left(-\frac{1}{2}\right)\cdot \frac{\pi^2}{6}\frac{K}{M}(6a^2-6aM+M^2)\notag\\
		&=\tau \frac{1}{48i\pi M}(6a^2-6aM+M^2).
	\end{align}
	Recalling that $a=1,2,\ldots,M$, we have
	\begin{align*}
		|\Re({-2}\pi i R_{12})|= |\Re(\tau)|\frac{|6a^2-6aM+M^2|}{24M}\le \frac{1}{X}\cdot \frac{M^2}{24M}=\frac{1}{24}MX^{-1}.
	\end{align*}
	Hence, \eqref{eq:R12} is proved.
	
	To estimate $R_{13}$, we proceed as follows. For \textit{Cases 1 \& 2} as presented in Section \ref{sec:res-lemmas}, $\lambda$ cannot take the value $K$, and thus we always have the vanishing of $\sum\limits_{1\le\nu\le K}\cos\frac{2\pi \nu \lambda}{K}$. It follows that
	\begin{align}
		R_{13}=0\qquad{\text{(\textit{Cases 1 \& 2})}}.
	\end{align}
	For \textit{Case 3}, the only contribution comes from the marginal case $(\lambda,\rho)=(K,k)$. Therefore,
	\begin{align}
		R_{13}=-\frac{\log z}{4\pi i}\qquad{\text{(\textit{Case 3})}}.
	\end{align}
	Notice that $\Re(\log z)=\log|z|=\log|\frac{\tau k}{2\pi}|$. Since $\frac{1}{X}\le |\tau|\le \frac{2\sqrt{2}\pi}{kN}$ and $1\le k\le N$ while $N=\lfloor\sqrt{2\pi X}\rfloor>\sqrt{2}$ as we have assumed that $X\ge 16$, it follows that $\frac{1}{2\pi X}\le |\frac{\tau k}{2\pi}|<1$. Thus,
	\begin{align}\label{eq:Re-log-z}
		|\Re(\log z)|\le \log(2\pi X).
	\end{align}
	Overall,
	\begin{align*}
		|\Re({-2}\pi i R_{13})|\le \frac{1}{2}\log(2\pi X),
	\end{align*}
	which gives \eqref{eq:R13}.
	
	The estimation of $R_{14}$ is more delicate. For \textit{Case 1}, we simply have
	\begin{align}
		R_{14}=0\qquad{\text{(\textit{Case 1})}}.
	\end{align}
	For \textit{Case 2}, we have a marginal contribution from $(\lambda,\rho)=(\lambda_k,k)$. Notice that there is no contribution from the term $-\dg{\frac{\mu}{k}}$ as $\lambda_k\ne K$ implies the vanishing of $\sum\limits_{1\le\nu\le K}\cos\frac{2\pi \nu \lambda_k}{K}$. Thus,
	\begin{align*}
		R_{14}&=\frac{1}{4\pi i k K}\left(\sum_{1\le\mu\le k} 1\right)\!\!\left(\sum_{1\le\nu\le K}\cos\frac{2\pi \nu \lambda_k}{K}\dg{\frac{\nu}{K}}\right)\\
		\text{\tiny (by $\eqref{eq:sum-cos-gamma}$)}&=\frac{1}{4\pi i k K}\cdot k\cdot K \log\left(2\sin\frac{\pi \lambda_k}{K}\right).
	\end{align*}
	Therefore,
	\begin{align}
		R_{14}=\frac{1}{4\pi i}\log\left(2\sin\frac{\pi \lambda_k}{K}\right)\qquad{\text{(\textit{Case 2})}}.
	\end{align}
	For \textit{Case 3}, we have a marginal contribution from $(\lambda,\rho)=(K,k)$. Then
	\begin{align*}
		R_{14}&=\frac{1}{4\pi i k K}\left[\left(\sum_{1\le\mu\le k} 1\right)\!\!\left(\sum_{1\le\nu\le K}\dg{\frac{\nu}{K}}\right)-\left(\sum_{1\le\mu\le k} \dg{\frac{\mu}{k}}\right)\!\!\left(\sum_{1\le\nu\le K}1\right)\right]\\
		\text{\tiny (by $\eqref{eq:sum-gamma}$)}&=\frac{1}{4\pi i k K}\big[{-k}\cdot K(\gamma+\log K) + k(\gamma+\log k)\cdot K\big].
	\end{align*}
	Thus,
	\begin{align}
		R_{14}=\frac{1}{4\pi i}\log\frac{k}{K}=\frac{1}{4\pi i}\log\frac{M^*}{M}\qquad{\text{(\textit{Case 3})}}.
	\end{align}
	Notice that
	\begin{align}\label{eq:log-sin-bound}
		\left|\log(2\sin x)\right|\le \begin{cases}
			\dfrac{\pi \log 2}{2x} & \text{if $0<x\le \frac{\pi}{2}$},\\[10pt]
			\dfrac{\pi \log 2}{2(\pi-x)} & \text{if $\frac{\pi}{2}\le x< \pi$}.
		\end{cases}
	\end{align}
	Since $\frac{M^*}{M}\le \frac{\lambda_k}{K}\le 1-\frac{M^*}{M}$, we have
	\begin{align}\label{eq:log-sin-lambda}
		\left|\log\left(2\sin\frac{\pi \lambda_k}{K}\right)\right|\le \frac{\pi \log 2}{2}\cdot \frac{1}{\pi}\frac{M}{M^*}\le \frac{M}{M^*}.
	\end{align}
	Also, since $\log x\le x$ for $x\ge 1$, we have
	\begin{align*}
		\left|\log\frac{M^*}{M}\right|\le \frac{M}{M^*}.
	\end{align*}
	Overall,
	\begin{align*}
		|\Re({-2}\pi i R_{14})|\le \frac{1}{2}\max\left\{0,\left|\log\left(2\sin\frac{\pi \lambda_k}{K}\right)\right|,\left|\log\frac{M^*}{M}\right|\right\}\le \frac{M}{2M^*},
	\end{align*}
	as desired.
\end{proof}

\subsection{Evaluation of $R_2$}

We have, by also recalling \eqref{eq:sum=sum},
\begin{align*}
	R_2 = -\frac{1}{4\pi i k K} \sum_{\substack{1\le \rho\le k\\ \rho\equiv b\bmod{M^*}}}\sum_{\substack{1\le\mu\le k\\1\le\nu\le K}} \cos\frac{2\pi \mu\rho}{k}\sin\frac{2\pi \nu \lambda}{K}\cdot \cR_s.
\end{align*}
Due to the vanishing of $\sum\limits_{1\le\nu\le K}\sin\frac{2\pi \nu \lambda}{K}$, it is sufficient to merely consider terms in $\cR_s$ that involve $\nu$. Thus,
\begin{align*}
	R_2=R_{21}+R_{22}+R_{23},
\end{align*}
where
\begin{align*}
	R_{21}&:=-\frac{1}{4\pi i k K} \sum_{\substack{1\le \rho\le k\\ \rho\equiv b\bmod{M^*}}}\sum_{\substack{1\le\mu\le k\\1\le\nu\le K}} \cos\frac{2\pi \mu\rho}{k}\sin\frac{2\pi \nu \lambda}{K}\cdot\frac{2}{\pi}\dg{\frac{\mu}{k}}\dg{\frac{\nu}{K}},\\
	R_{22}&:=-\frac{1}{4\pi i k K} \sum_{\substack{1\le \rho\le k\\ \rho\equiv b\bmod{M^*}}}\sum_{\substack{1\le\mu\le k\\1\le\nu\le K}} \cos\frac{2\pi \mu\rho}{k}\sin\frac{2\pi \nu \lambda}{K}\cdot \frac{2\log z}{\pi}\dg{\frac{\nu}{K}},\\
	R_{23}&:=-\frac{1}{4\pi i k K} \sum_{\substack{1\le \rho\le k\\ \rho\equiv b\bmod{M^*}}}\sum_{\substack{1\le\mu\le k\\1\le\nu\le K}} \cos\frac{2\pi \mu\rho}{k}\sin\frac{2\pi \nu \lambda}{K}\cdot \frac{2}{\pi}\gamma_1\!\left(\frac{\nu}{K}\right).
\end{align*}

\begin{theorem}
	We have
	\begin{align}
		|\Re({-2}\pi i R_{21})|&\le 0.44X^{\frac{1}{2}}\log X + 1.3X^{\frac{1}{2}} + 0.25\log X+0.75\ll X^{\frac{1}{2}}\log X,\label{eq:R21}\\
		|\Re({-2}\pi i R_{22})|&\le \tfrac{1}{2}\log X+0.92\ll \log X,\label{eq:R22}\\
		|\Re({-2}\pi i R_{23})|&\le \tfrac{1}{4}\log X+\tfrac{1}{2}\tfrac{M}{(k,M)}+\tfrac{1}{2}\log\tfrac{M}{(k,M)}+\log\Gamma(\tfrac{(k,M)}{M})+2.59\ll \log X.\label{eq:R23}
	\end{align}
\end{theorem}

\begin{proof}
	To evaluate $R_{21}$, we first define
	\begin{align*}
		R_{211}:=-\frac{1}{4\pi i k K} \sum_{\substack{1\le \rho< k\\ \rho\equiv b\bmod{M^*}}}\sum_{\substack{1\le\mu\le k\\1\le\nu\le K}} \cos\frac{2\pi \mu\rho}{k}\sin\frac{2\pi \nu \lambda}{K}\cdot\frac{2}{\pi}\dg{\frac{\mu}{k}}\dg{\frac{\nu}{K}}.
	\end{align*}
	Here the difference between $R_{21}$ and $R_{211}$ is that in the latter, the range of $\rho$ is reduced to $1\le \rho<k$. For \textit{Case 1} as presented in Section \ref{sec:res-lemmas}, $\rho$ cannot take the value $k$. For \textit{Case 3}, although $\rho$ may take the value $k$, the corresponding $\lambda$ of $\rho=k$ is $K$, so we have the vanishing of $\sin\frac{2\pi \nu K}{K}$. Thus,
	\begin{align}
		R_{21}=R_{211}\qquad{\text{(\textit{Cases 1 \& 3})}}.
	\end{align}
	For \textit{Case 2}, we have an additional contribution from $(\lambda,\rho)=(\lambda_k,k)$, which equals
	\begin{align*}
		R_{212}:=-\frac{1}{4\pi i k K} \sum_{\substack{1\le\mu\le k\\1\le\nu\le K}} \sin\frac{2\pi \nu \lambda_k}{K}\cdot\frac{2}{\pi}\dg{\frac{\mu}{k}}\dg{\frac{\nu}{K}}.
	\end{align*}
	Thus,
	\begin{align}
		R_{21}=R_{211}+R_{212}\qquad{\text{(\textit{Case 2})}}.
	\end{align}
	Notice that
	\begin{align}\label{eq:R211-eva}
		R_{211}&=-\frac{1}{2i \pi^2 k K} \sum_{\substack{1\le \rho< k\\ \rho\equiv b\bmod{M^*}}}\sum_{1\le\mu\le k} \cos\frac{2\pi \mu\rho}{k}\dg{\frac{\mu}{k}}\sum_{1\le\nu\le K}\sin\frac{2\pi \nu \lambda}{K}\dg{\frac{\nu}{K}}\notag\\
		\text{\tiny (by $\substack{\eqref{eq:sum-cos-gamma}\\ \eqref{eq:sum-sin-gamma}}$)}&=-\frac{1}{2i \pi^2 k K} \sum_{\substack{1\le \rho< k\\ \rho\equiv b\bmod{M^*}}}k \log\left(2\sin\frac{\pi \rho}{k}\right)\cdot \frac{\pi}{2}(2\lambda-K)\notag\\
		&=-\frac{1}{4\pi i K}\sum_{\substack{1\le \rho< k\\ \rho\equiv b\bmod{M^*}}}(2\lambda-K)\log\left(2\sin\frac{\pi \rho}{k}\right).
	\end{align}
	In light of \eqref{eq:log-sin-bound} together with the fact that $|2\lambda-K|\le K$, we have
	\begin{align*}
		|\Re({-2}\pi i R_{211})|&\le \frac{1}{2K}\cdot K\cdot 2\sum_{1\le \rho< k}\frac{\pi \log 2}{2}\frac{k}{\pi \rho}\\
		&=\frac{\log 2}{2}k\sum_{1\le \rho< k}\frac{1}{\rho}\\
		&\le \frac{\log 2}{2}\sqrt{2\pi X}(\log \sqrt{2\pi X}+\gamma).
	\end{align*}
	Also,
	\begin{align}
		R_{212}&=-\frac{1}{2i \pi^2 k K} \sum_{1\le\mu\le k} \dg{\frac{\mu}{k}}\sum_{1\le\nu\le K}\sin\frac{2\pi \nu \lambda_k}{K}\dg{\frac{\nu}{K}}\notag\\
		\text{\tiny (by $\substack{\eqref{eq:sum-sin-gamma}\\ \eqref{eq:sum-gamma}}$)}&= -\frac{1}{2i \pi^2 k K} \cdot \big({-k}(\gamma+\log k)\big) \cdot \frac{\pi}{2}(2\lambda_k-K)\notag\\
		&=\frac{1}{4\pi i K}(\gamma+\log k)(2\lambda_k-K).
	\end{align}
	Thus,
	\begin{align*}
		|\Re({-2}\pi i R_{211})|\le \frac{1}{2}(\log \sqrt{2\pi X}+\gamma).
	\end{align*}
	Overall,
	\begin{align*}
		|\Re({-2}\pi i R_{21})|\le \frac{\log 2}{2}\sqrt{2\pi X}(\log \sqrt{2\pi X}+\gamma)+\frac{1}{2}(\log \sqrt{2\pi X}+\gamma),
	\end{align*}
	thereby confirming \eqref{eq:R21}.
	
	For $R_{22}$, it is also easily seen that
	\begin{align}
		R_{22}=0\qquad{\text{(\textit{Cases 1 \& 3})}}.
	\end{align}
	Considering \textit{Case 2}, there is an extra contribution from $(\lambda,\rho)=(\lambda_k,k)$, given by
	\begin{align*}
		R_{22}&=-\frac{\log z}{2i \pi^2 k K} \left(\sum_{1\le\mu\le k} 1\right)\!\!\left(\sum_{1\le\nu\le K}\sin\frac{2\pi \nu \lambda_k}{K}\dg{\frac{\nu}{K}}\right)\\
		\text{\tiny (by $\eqref{eq:sum-sin-gamma}$)}&=-\frac{\log z}{2i \pi^2 k K}\cdot k\cdot \frac{\pi}{2}(2\lambda_k-K).
	\end{align*}
	Thus,
	\begin{align}
		R_{22}=-\frac{2\lambda_k-K}{4\pi i K}\log z\qquad{\text{(\textit{Case 2})}}.
	\end{align}
	Overall, we have, by also recalling \eqref{eq:Re-log-z},
	\begin{align*}
		|\Re({-2}\pi i R_{22})|\le \frac{1}{2K}\cdot K\cdot |\Re(\log z)|\le \frac{1}{2}\log(2\pi X),
	\end{align*}
	which gives \eqref{eq:R22}.
	
	In the same vein, it holds that
	\begin{align}
		R_{23}=0\qquad{\text{(\textit{Cases 1 \& 3})}}.
	\end{align}
	For \textit{Case 2}, we get from the extra contribution from $(\lambda,\rho)=(\lambda_k,k)$ that
	\begin{align*}
		R_{23}&=-\frac{1}{2i \pi^2 k K} \left(\sum_{1\le\mu\le k} 1\right)\!\!\left(\sum_{1\le\nu\le K}\sin\frac{2\pi \nu \lambda_k}{K}\gamma_1\!\left(\frac{\nu}{K}\right)\right)\\
		\text{\tiny (by $\eqref{eq:sum-Sti}$)}&=-\frac{1}{2i \pi^2 k K}\cdot k\cdot \Bigg[\frac{\pi}{2}\big(\gamma+\log(2\pi K)\big)(2\lambda_k-K)-\frac{\pi \log(2\pi)}{2}K\notag\\
		&\quad+\frac{\pi K}{2}\log\left(2\sin\frac{\pi \lambda_k}{K}\right)+\pi K \log \Gamma\left(\frac{\lambda_k}{K}\right)\Bigg].
	\end{align*}
	Thus,
	\begin{align}
		R_{23}&=-\frac{1}{4\pi i K}\Bigg[\big(\gamma+\log(2\pi K)\big)(2\lambda_k-K)-K\log(2\pi)\notag\\
		&\quad+K\log\left(2\sin\frac{\pi \lambda_k}{K}\right)+2K\log \Gamma\left(\frac{\lambda_k}{K}\right)\Bigg]\qquad{\text{(\textit{Case 2})}}.
	\end{align}
	Overall,
	\begin{align*}
		|\Re({-2}\pi i R_{23})|&\le \frac{1}{2}\Bigg[\gamma+\log(2\pi K)+\log(2\pi)+\left|\log\left(2\sin\frac{\pi \lambda_k}{K}\right)\right|+2\left|\log \Gamma\left(\frac{\lambda_k}{K}\right)\right|\Bigg]\\
		&\le \frac{1}{2}\Bigg[\gamma+2\log(2\pi)+\log \sqrt{2\pi X}+\log\frac{M}{M^*}+\frac{M}{M^*}+2\log\Gamma\left(\frac{M^*}{M}\right)\Bigg],
	\end{align*}
	where we recall that $K=k\frac{M}{M^*}$ with $1\le k\le \sqrt{2\pi X}$ and that $\frac{M^*}{M}\le \frac{\lambda_k}{K}\le 1-\frac{M^*}{M}$, while \eqref{eq:log-sin-lambda} is also used. Therefore, we arrive at \eqref{eq:R23}.
\end{proof}

\subsection{Evaluation of $R_3$}

We have
\begin{align*}
	R_3 = \frac{1}{4\pi k K} \sum_{\substack{1\le \rho\le k\\ \rho\equiv b\bmod{M^*}}} \sum_{\substack{1\le\mu\le k\\1\le\nu\le K}} \sin\frac{2\pi \mu\rho}{k}\sin\frac{2\pi \nu \lambda}{K}\cdot \cR_s.
\end{align*}
We may further reduce the range of $\rho$ to $1\le \rho<k$ due to the vanishing of $\sin\frac{2\pi \mu\rho}{k}$ when $\rho=k$. Meanwhile, according to Remark \ref{rmk:3-cases}, for each $\rho$ in this reduced range, the corresponding $\lambda$ is not $K$. Thus, we only need to take into account terms in $\cR_s$ that involve both $\mu$ and $\nu$. It follows that
\begin{align*}
	R_3 = \frac{1}{4\pi k K} \sum_{\substack{1\le \rho< k\\ \rho\equiv b\bmod{M^*}}} \sum_{\substack{1\le\mu\le k\\1\le\nu\le K}} \sin\frac{2\pi \mu\rho}{k}\sin\frac{2\pi \nu \lambda}{K}\cdot \frac{2}{\pi}\dg{\frac{\mu}{k}}\dg{\frac{\nu}{K}}.
\end{align*}
Since this summation is purely real, we directly obtain the following relation.

\begin{theorem}
	We have
	\begin{align}
		|\Re({-2}\pi i R_{3})|=0.\label{eq:R3}
	\end{align}
\end{theorem}

\begin{remark}
	It is worth pointing out that $R_{3}$ may be further evaluated as follows:
	\begin{align}
		R_3&=\frac{1}{2\pi^2 k K} \sum_{\substack{1\le \rho< k\\ \rho\equiv b\bmod{M^*}}} \sum_{1\le \mu\le k} \sin\frac{2\pi \mu\rho}{k} \dg{\frac{\mu}{k}}\sum_{1\le \nu \le K}\sin\frac{2\pi \nu \lambda}{K} \dg{\frac{\nu}{K}}\notag\\
		\text{\tiny (by $\eqref{eq:sum-sin-gamma}$)}&=\frac{1}{2\pi^2 k K} \sum_{\substack{1\le \rho< k\\ \rho\equiv b\bmod{M^*}}} \frac{\pi}{2}(2\rho-k)\cdot \frac{\pi}{2}(2\lambda-K)\notag\\
		&=\frac{1}{8 k K} \sum_{\substack{1\le \rho< k\\ \rho\equiv b\bmod{M^*}}} (2\rho-k)(2\lambda-K).
	\end{align}
\end{remark}

\subsection{Evaluation of $R_4$}

We have
\begin{align*}
	R_4=-\frac{1}{4\pi k K} \sum_{\substack{1\le \rho\le k\\ \rho\equiv b\bmod{M^*}}} \sum_{\substack{1\le\mu\le k\\1\le\nu\le K}} \sin\frac{2\pi \mu\rho}{k}\cos\frac{2\pi \nu \lambda}{K}\cdot \cR_c.
\end{align*}
Similar to the analysis for $R_3$, it holds that
\begin{align*}
	R_4=R_{41}+R_{42},
\end{align*}
where
\begin{align*}
	R_{41}&:=-\frac{1}{4\pi k K} \sum_{\substack{1\le \rho< k\\ \rho\equiv b\bmod{M^*}}} \sum_{\substack{1\le\mu\le k\\1\le\nu\le K}} \sin\frac{2\pi \mu\rho}{k}\cos\frac{2\pi \nu \lambda}{K}\cdot \frac{2\zrie{2,\frac{\mu}{k}}\zrie{0,\frac{\nu}{K}}}{\pi z},\\
	R_{42}&:=-\frac{1}{4\pi k K} \sum_{\substack{1\le \rho< k\\ \rho\equiv b\bmod{M^*}}} \sum_{\substack{1\le\mu\le k\\1\le\nu\le K}} \sin\frac{2\pi \mu\rho}{k}\cos\frac{2\pi \nu \lambda}{K}\cdot \left(-\frac{2z\zrie{0,\frac{\mu}{k}}\zrie{2,\frac{\nu}{K}}}{\pi}\right).
\end{align*}

\begin{theorem}
	We have
	\begin{equation}\label{eq:R41}
		{-2}\pi i R_{41} =\begin{cases}
			0 & \text{if $b= (k,M)$},\\
			-\frac{1}{\tau}\frac{(k,M)^2}{M}\frac{2\pi i}{k^2}\bigg(\zeta'\left(-1,\frac{b}{(k,M)}\right)-\zeta'\left(-1,\frac{(k,M)-b}{(k,M)}\right)\bigg) & \text{if $b\ne (k,M)$}.
		\end{cases}
	\end{equation}
	Also,
	\begin{equation}\label{eq:R42}
		|\Re({-2}\pi i R_{42})|\le \tfrac{1}{12}\tfrac{M}{(k,M)}\log X +0.25\tfrac{M}{(k,M)}\ll \log X.
	\end{equation}
\end{theorem}

\begin{proof}
	For $R_{41}$, we have
	\begin{align*}
		R_{41}&=-\frac{1}{z}\frac{1}{2\pi^2k K}\sum_{\substack{1\le \rho< k\\ \rho\equiv b\bmod{M^*}}} \sum_{1\le\mu\le k}\sin\frac{2\pi \mu\rho}{k}\zrie{2,\frac{\mu}{k}} \sum_{1\le\nu\le K}\cos\frac{2\pi \nu \lambda}{K}\zrie{0,\frac{\nu}{K}}\\
		\text{\tiny (by $\substack{\eqref{eq:sum-cos-zeta-0}\\ \eqref{eq:sum-sin-zeta-2}}$)}&=-\frac{1}{z}\frac{1}{2\pi^2k K}\sum_{\substack{1\le \rho< k\\ \rho\equiv b\bmod{M^*}}} 2\pi k^2\Bigg(\zeta'\left(-1,\frac{\rho}{k}\right)-\zeta'\left(-1,1-\frac{\rho}{k}\right)\Bigg)\cdot \left(-\frac{1}{2}\right)\\
		&=\frac{1}{z}\frac{k}{2\pi K}\sum_{\substack{1\le \rho< k\\ \rho\equiv b\bmod{M^*}}} \Bigg(\zeta'\left(-1,\frac{\rho}{k}\right)-\zeta'\left(-1,1-\frac{\rho}{k}\right)\Bigg).
	\end{align*}
	If $b=M^*$, then both $\rho$ and $k-\rho$ run through all multiples of $M^*$ within the range $[1,k)$, and thus,
	\begin{equation}\label{eq:R41-case1}
		R_{41} = 0.
	\end{equation}
	We further notice that if $d\mid k$ and $1\le c\le d$, then for any $s\ne 1$,
	\begin{equation*}
		\sum_{\substack{1\le \ell\le k\\\ell\equiv c\bmod{d}}}\zrie{s,\frac{\ell}{k}}=\left(\frac{k}{d}\right)^s\zrie{s,\frac{c}{d}}.
	\end{equation*}
	Therefore,
	\begin{equation*}
		\sum_{\substack{1\le \ell< k\\\ell\equiv c\bmod{d}}}\zeta'\left(s,\frac{\ell}{k}\right)=\left(\frac{k}{d}\right)^s\zrie{s,\frac{c}{d}}\log\left(\frac{k}{d}\right)+\left(\frac{k}{d}\right)^s\zeta'\left(s,\frac{c}{d}\right).
	\end{equation*}
	As $M^*=(k,M)$ divides $k$, it follows that if $b\ne M^*$ so that $\rho\ne k$,
	\begin{align}\label{eq:R41-case2}
		R_{41}&=\frac{1}{z}\frac{k}{2\pi K}\Bigg(\left(\frac{M^*}{k}\right)\zrie{-1,\frac{b}{M^*}}\log\frac{k}{M^*}+\left(\frac{M^*}{k}\right)\zeta'\left(-1,\frac{b}{M^*}\right)\notag\\
		&\quad-\left(\frac{M^*}{k}\right)\zrie{-1,\frac{M^*-b}{M^*}}\log\frac{k}{M^*}-\left(\frac{M^*}{k}\right)\zeta'\left(-1,\frac{M^*-b}{M^*}\right)\Bigg)\notag\\
		&=\frac{1}{z}\frac{(k,M)^2}{M}\frac{1}{2\pi k}\Bigg(\zeta'\left(-1,\frac{b}{M^*}\right)-\zeta'\left(-1,\frac{M^*-b}{M^*}\right)\Bigg)\notag\\
		&=\frac{1}{\tau}\frac{(k,M)^2}{M}\frac{1}{k^2}\Bigg(\zeta'\left(-1,\frac{b}{(k,M)}\right)-\zeta'\left(-1,\frac{(k,M)-b}{(k,M)}\right)\Bigg).
	\end{align}
	Combining the two cases \eqref{eq:R41-case1} and \eqref{eq:R41-case2} gives \eqref{eq:R41}.
	
	To evaluate $R_{42}$, we find that
	\begin{align}
		R_{42}&= z \frac{1}{2\pi^2k K}\sum_{\substack{1\le \rho< k\\ \rho\equiv b\bmod{M^*}}} \sum_{1\le\mu\le k}\sin\frac{2\pi \mu\rho}{k}\zrie{0,\frac{\mu}{k}} \sum_{1\le\nu\le K}\cos\frac{2\pi \nu \lambda}{K}\zrie{2,\frac{\nu}{K}}\notag\\
		\text{\tiny (by $\substack{\eqref{eq:sum-cos-zeta-2}\\ \eqref{eq:sum-sin-zeta-0}}$)}&=z \frac{1}{2\pi^2k K}\sum_{\substack{1\le \rho< k\\ \rho\equiv b\bmod{M^*}}} \frac{1}{2}\cot\frac{\pi\rho}{k} \cdot \frac{\pi^2}{6}(6\lambda^2-6K\lambda+K^2)\notag\\
		&=\tau\frac{1}{48\pi K}\sum_{\substack{1\le \rho< k\\ \rho\equiv b\bmod{M^*}}} (6\lambda^2-6K\lambda+K^2) \cot\frac{\pi\rho}{k}.
	\end{align}
	Notice that $|6\lambda^2-6K\lambda+K^2|\le K^2$ and that $K=k\frac{M}{M^*}$ with $1\le k\le N\le \sqrt{2\pi X}$. Also, for $0<x< \pi$,
	$$|\cot x|\le \frac{1}{x}+\frac{1}{\pi-x}.$$
	Thus,
	\begin{align*}
		|\Re({-2}\pi i R_{42})|&=|\Im(\tau)|\frac{1}{24K}\left|\sum_{\substack{1\le \rho< k\\ \rho\equiv b\bmod{M^*}}} (6\lambda^2-6K\lambda+K^2) \cot\frac{\pi\rho}{k}\right|\\
		&\le \frac{2\pi}{kN} \frac{1}{24K}\cdot K^2\cdot 2\sum_{1\le \ell<k}\frac{k}{\pi \ell}\\
		&\le \frac{1}{6}\frac{M}{M^*}(\log \sqrt{2\pi X}+\gamma),
	\end{align*}
	which gives \eqref{eq:R42}. 
\end{proof}

\section{Explicit bounds of $G(q)$ on the minor arcs}\label{sec:G-minor}

Recall that
$$G(q)=\frac{(q^3;q^4)_\infty}{(q;q^4)_\infty (q^6;q^8)_\infty}.$$
The object of this section is the following uniform bound of $|G(q)|$ when $q$ is away from $\pm 1$.

\begin{theorem}\label{th:bound-away}
	Let $\mathcal{Q}_{h/k}$ be as in Remark \ref{rmk:Q_h/k}. For any $q$ with $|q|=e^{-\frac{1}{X}}$ such that it is not in $\mathcal{Q}_{1/1}$ and $\mathcal{Q}_{1/2}$, we have, if $X\ge 3.4\times 10^7$, then
	\begin{align}\label{eq:exp-bound}
	|G(q)|\le \exp\Bigg(\left(\frac{\pi^2}{48}-\frac{1}{100}\right)X\Bigg).
	\end{align}
	Also, if $q=e^{-\tau+\frac{2\pi i h}{k}}$ with $\tau=X^{-1}+2\pi i Y$ is in $\mathcal{Q}_{1/1}$ or $\mathcal{Q}_{1/2}$, then \eqref{eq:exp-bound} still holds under the assumption $X\ge 3.4\times 10^7$ provided that $|Y|\ge \frac{1}{2\pi X}$.
\end{theorem}

Noticing that $\tau=X^{-1}+2\pi i Y$, we have
\begin{align}\label{eq:tau^-1}
\tau^{-1}=\frac{X^{-1}}{X^{-2}+4\pi^2 Y^2}-i\frac{2\pi Y}{X^{-2}+4\pi^2 Y^2}.
\end{align}
In the sequel, $b$ is written as $b(h,a,k,M)$ to avoid confusion. We also define
\begin{align}
\Main_{a,M}&:=\frac{1}{\tau}\frac{(k,M)^2}{k^2 M}\Bigg(\pi^2\left(\frac{ b^2}{(k,M)^2}-\frac{ b}{(k,M)}+\frac{1}{6}\right)\nonumber\\
&\;\quad+2\pi i\left(-\zeta'\left(-1,\frac{b}{(k,M)}\right)+\zeta'\left(-1,\frac{b^*}{(k,M)}\right)\right)\Bigg),
\end{align}
which is the main term in \eqref{eq:th-main}. Further,
\begin{equation}
\Main_G:=\Main_{1,4}-\Main_{3,4}+\Main_{6,8}
\end{equation}
denotes the main term of $\log G(q)$, whereas
\begin{equation}
E_G:=\log G(q)-\Main_G
\end{equation}
denotes the error term.

\subsection{Case 1: $k\in 2\mathbb{N}+1$}

Notice that $(k,4)=1$. Hence, we always have $b(h,1,k,4)=b(h,3,k,4)=1$. Also, $(k,8)=1$ so that $b(h,6,k,8)=1$. It is not hard to compute that
\begin{align}\label{eq:M_1}
\Main_G = \frac{1}{\tau}\frac{\pi^2}{48 k^2}.
\end{align}
By \eqref{eq:tau^-1},
\begin{equation}
\Re(\Main_G)\le \frac{\pi^2}{48 k^2} X.
\end{equation}
We may also compute from the bounds in \eqref{eq:bound-integral}, \eqref{eq:R12}, \eqref{eq:R13}, \eqref{eq:R14}, \eqref{eq:R21}, \eqref{eq:R22}, \eqref{eq:R23}, \eqref{eq:R3} and \eqref{eq:R42} that
\begin{align}
|\Re(E_G)|\le 1.32X^{\frac{1}{2}}\log X+506.83 X^{\frac{1}{2}}+5.84\log X+79.54+0.67X^{-1}.
\end{align}

\subsection{Case 2: $k\in 4\mathbb{N}+2$}

Notice that $(k,4)=2$. Since $(h,k)=1$ while $k$ is even, $h$ is then odd. Hence, we always have $b(h,1,k,4)=b(h,3,k,4)=1$. Also, $(k,8)=1$ and thus $b(h,6,k,8)=2$. It is not hard to compute that
\begin{align}\label{eq:M_2}
\Main_G = \frac{1}{\tau}\frac{\pi^2}{12 k^2}.
\end{align}
Further, we derive from \eqref{eq:tau^-1} that
\begin{equation}
\Re(\Main_G)\le \frac{\pi^2}{12 k^2} X.
\end{equation}
For the error term $E_G$, we have
\begin{align}
|\Re(E_G)|\le 1.32X^{\frac{1}{2}}\log X+92.81 X^{\frac{1}{2}}+5.17\log X+38.09+0.67X^{-1}.
\end{align}

\subsection{Case 3: $k\in 8\mathbb{N}+4$}

Notice that $(k,4)=4$. If $h\equiv 1 \pmod{4}$, then $b(h,1,k,4)=3$ and $b(h,3,k,4)=1$. If $h\equiv 3 \pmod{4}$, then $b(h,1,k,4)=1$ and $b(h,3,k,4)=3$. Therefore,
\begin{align*}
\Main_{1,4}-\Main_{3,4}=\frac{1}{\tau}\frac{16\pi i \chi(h)}{k^2}\left(\zeta'\left(-1,\tfrac{1}{4}\right)-\zeta'\left(-1,\tfrac{3}{4}\right)\right),
\end{align*}
where
$$\chi(h)=\begin{cases}
1 & \text{if $h\equiv 1 \pmod{4}$},\\
-1 & \text{if $h\equiv 3 \pmod{4}$}.
\end{cases}$$
Also, $(k,8)=4$ so that $k$ is even. Then $h$ is odd as $(h,k)=1$. Now we have $b(h,6,k,8)=2$. It follows that
\begin{equation*}
\Main_{6,8}=-\frac{1}{\tau}\frac{\pi^2}{6k^2}.
\end{equation*}
As a consequence,
\begin{align}
\Main_G = \frac{1}{\tau}\Bigg({-\frac{\pi^2}{6k^2}}+\frac{16\pi i \chi(h)}{k^2}\left(\zeta'\left(-1,\tfrac{1}{4}\right)-\zeta'\left(-1,\tfrac{3}{4}\right)\right)\Bigg).
\end{align}
On the other hand, by \eqref{eq:tau^-1},
\begin{align*}
\Re(\Main_G)&=-\frac{\pi^2}{6k^2}\frac{X^{-1}}{X^{-2}+4\pi^2 Y^2}+\frac{16\pi \chi(h)}{k^2}\left(\zeta'\left(-1,\tfrac{1}{4}\right)-\zeta'\left(-1,\tfrac{3}{4}\right)\right)\frac{2\pi Y}{X^{-2}+4\pi^2 Y^2}\\
&\le \frac{1}{k^2}\cdot \frac{-\frac{\pi^2}{6}X^{-1}+16\pi \left(\zeta'\left(-1,\frac{1}{4}\right)-\zeta'\left(-1,\frac{3}{4}\right)\right)2\pi |Y|}{X^{-2}+4\pi^2 |Y|^2}\\
&= \frac{\pi^2}{6 k^2}\cdot \frac{-X^{-1}+192\left(\zeta'\left(-1,\frac{1}{4}\right)-\zeta'\left(-1,\frac{3}{4}\right)\right)|Y|}{X^{-2}+4\pi^2 |Y|^2}.
\end{align*}
We claim that
\begin{align}\label{eq:G-case-3}
\Re(\Main_G)\le \frac{2.94}{k^2}X.
\end{align}
For this purpose, it is sufficient to prove that
$$\frac{\pi^2}{6 k^2}\cdot \frac{-X^{-1}+192\left(\zeta'\left(-1,\frac{1}{4}\right)-\zeta'\left(-1,\frac{3}{4}\right)\right)|Y|}{X^{-2}+4\pi^2 |Y|^2}\le \frac{2.94}{k^2}X.$$
Namely,
\begin{align*}
70.56 X |Y|^2 - 192\left(\zeta'\left(-1,\tfrac{1}{4}\right)-\zeta'\left(-1,\tfrac{3}{4}\right)\right)|Y|+\left(\frac{17.64}{\pi^2}+1\right)X^{-1}\ge 0.
\end{align*}
Notice that on the left-hand side if we replace $|Y|$ by $t$ and treat it as a quadratic function of real $t$, then it reaches the minimum when
$$t=\frac{192\left(\zeta'\left(-1,\frac{1}{4}\right)-\zeta'\left(-1,\frac{3}{4}\right)\right)}{2\times 70.56 X}.$$
Now \eqref{eq:G-case-3} holds as the minimum is
$$-70.56 X \left(\frac{192\left(\zeta'\left(-1,\frac{1}{4}\right)-\zeta'\left(-1,\frac{3}{4}\right)\right)}{2\times 70.56 X}\right)^2+\left(\frac{17.64}{\pi^2}+1\right)X^{-1}\ge 0.01X^{-1}\ge 0.$$
Finally, for the error term $E_G$, we have
\begin{align}
|\Re(E_G)|\le 1.32X^{\frac{1}{2}}\log X+19.62 X^{\frac{1}{2}}+4.84\log X+23.24+0.67X^{-1}.
\end{align}

\subsection{Case 4: $k\in 8\mathbb{N}+8$}

As in \textit{Case 3}, we still have
\begin{align*}
\Main_{1,4}-\Main_{3,4}=\frac{1}{\tau}\frac{16\pi i \chi(h)}{k^2}\left(\zeta'\left(-1,\tfrac{1}{4}\right)-\zeta'\left(-1,\tfrac{3}{4}\right)\right).
\end{align*}
Also, $(k,8)=8$. If $h\equiv 1 \pmod{4}$, then $b(h,6,k,8)=2$, while if $h\equiv 3 \pmod{4}$, then $b(h,6,k,8)=6$. Hence,
\begin{align*}
\Main_{6,8}=\frac{1}{\tau}\Bigg({-\frac{\pi^2}{6k^2}}-\frac{16\pi i \chi(h)}{k^2}\left(\zeta'\left(-1,\tfrac{1}{4}\right)-\zeta'\left(-1,\tfrac{3}{4}\right)\right)\Bigg).
\end{align*}
In consequence,
\begin{equation}
\Main_G=-\frac{1}{\tau}\frac{\pi^2}{6k^2}.
\end{equation}
Further,
\begin{equation}
\Re(\Main_G)< 0.
\end{equation}
For the error term $E_G$, we have
\begin{align}
|\Re(E_G)|\le 1.32X^{\frac{1}{2}}\log X+12.04 X^{\frac{1}{2}}+4.75\log X+19.94+0.67X^{-1}.
\end{align}

\subsection{Proof of Theorem \ref{th:bound-away}}

We start by noticing that
\begin{align*}
	\log |G(q)|=\Re(\log G(q))\le \Re(\Main_G)+|\Re(E_G)|.
\end{align*}
Now the first part of Theorem \ref{th:bound-away} simply follows from a direct computation by taking into account the bounds for $\Re(\Main_G)$ and $|\Re(E_G)|$. For the second part, we notice from \eqref{eq:tau^-1} that, when $|Y|\ge \frac{1}{2\pi X}$,
$$\Re(\tau^{-1})\le \frac{X}{2}.$$
Whenever $q$ is in $\mathcal{Q}_{1/1}$ or $\mathcal{Q}_{1/2}$ with $|Y|\ge \frac{1}{2\pi X}$, we apply \eqref{eq:M_1} and \eqref{eq:M_2} to obtain the bound
$$\Re(\Main_G)\le \frac{\pi^2}{48}\frac{X}{2}.$$
Hence, \eqref{eq:exp-bound} follows by inserting the contribution of the error term and carrying on a routine computation.

\section{Precise approximations of $G(q)$ near the dominant poles}\label{sec:G-major}

Recall that
\begin{align}\label{eq:G(q)}
G(q) = \frac{1}{(q,-q^3;q^4)_\infty}.
\end{align}
From the analysis in the previous section, we know that $G(q)$ indeed has dominant poles at $q=\pm 1$. In fact, if $q=e^{-\tau+\frac{2\pi i h}{k}}$ is in $\mathcal{Q}_{1/1}$ or $\mathcal{Q}_{1/2}$, then \eqref{eq:M_1} and \eqref{eq:M_2} tell us that $\log G(q)$ is dominated by $\frac{\pi^2}{48\tau}$ while the coefficient $\frac{\pi^2}{48}$ is the largest in comparison with that of other $\mathcal{Q}_{h/k}$. Now we want to give some more precise approximations of $\log G(q)$ near the dominant poles.

\begin{theorem}
	Let $\tau=X^{-1}+2\pi i Y$ with $|Y|\le \frac{1}{2\pi X}$. Then
	\begin{equation}\label{eq:near:q=1}
	\log G(e^{-\tau})=\frac{\pi^2}{48}\frac{1}{\tau}-\frac{1}{4}\log \tau-\frac{3}{4}\log 2-\frac{1}{2}\log \pi+\log \Gamma(\tfrac{1}{4})+E_+,
	\end{equation}
	where
	\begin{equation}
	|E_+|\le 0.66 X^{-\frac{3}{4}}.
	\end{equation}
	Also,
	\begin{equation}\label{eq:near:q=-1}
	\log G(-e^{-\tau})=\frac{\pi^2}{48}\frac{1}{\tau}+\frac{1}{4}\log \tau-\frac{1}{4}\log 2-\frac{1}{2}\log \pi+\log \Gamma(\tfrac{3}{4})+E_-,
	\end{equation}
	where
	\begin{equation}
	|E_-|\le 0.82 X^{-\frac{3}{4}}.
	\end{equation}
\end{theorem}

\begin{remark}
	As pointed out by one of the referees, there are different approaches to finding these asymptotic behaviors. For instance, in the literature the Euler--Maclaurin summation formula is often used.
\end{remark}

\begin{proof}
	We deduce from \eqref{eq:G(q)} with the help of the inverse Mellin transform that
	\begin{align*}
	\log G(e^{-\tau})&=\sum_{m\ge 0}\sum_{\ell\ge 1}\Bigg(\frac{e^{-(4m+1)\ell\tau}}{\ell}+\frac{(-1)^\ell e^{-(4m+3)\ell\tau}}{\ell}\Bigg)\\
	&=\frac{1}{2\pi i}\int_{(\tfrac{3}{2})} \tau^{-s}\Gamma(s)\sum_{m\ge 0}\sum_{\ell\ge 1} \ell^{-s-1}\Bigg(\frac{1}{(4m+1)^s}+\frac{(-1)^\ell}{(4m+3)^s}\Bigg)ds\\
	&=\frac{1}{2\pi i}\int_{(\tfrac{3}{2})} (4\tau)^{-s}\Gamma(s)\zeta(s+1)\Big(\zrie{s,\tfrac{1}{4}}-(1-2^{-s})\zrie{s,\tfrac{3}{4}}\Big)ds\\
	&=:\frac{1}{2\pi i}\int_{(\tfrac{3}{2})} \Theta_+(s) ds.
	\end{align*}
	Now one may shift the path of integration to $(-\tfrac{3}{4})$ by taking into consideration of the residues of $\Theta_+(s)$ inside the strip $-\frac{3}{4}<\Re(s)<\frac{3}{2}$. Hence,
	\begin{align*}
	\log G(e^{-\tau}) = \sum_{-\frac{3}{4}<\Re(s)<\frac{3}{2}}\Res_s \Theta_+(s) +\frac{1}{2\pi i}\int_{(-\tfrac{3}{4})} \Theta_+(s) ds.
	\end{align*}
	Notice that $\Theta_+(s)$ has two singularities at $s=0$ and $1$, respectively, when $-\frac{3}{4}<\Re(s)<\frac{3}{2}$. We therefore compute that
	\begin{align*}
	\Res_{s=1} \Theta_+(s) = \frac{\pi^2}{48}\frac{1}{\tau},
	\end{align*}
	and that
	\begin{align*}
	\Res_{s=0} \Theta_+(s) &= -\frac{1}{4}\log(4\tau)+\zeta'\left(0,\tfrac{1}{4}\right)-(\log 2)\zrie{0,\tfrac{3}{4}}\\
	&=-\frac{1}{4}\log(4\tau)+\log \Gamma(\tfrac{1}{4})-\frac{1}{2}\log(2\pi)+\frac{1}{4}\log 2\\
	&=-\frac{1}{4}\log \tau-\frac{3}{4}\log 2-\frac{1}{2}\log \pi+\log \Gamma(\tfrac{1}{4}).
	\end{align*}
	Further, recalling that $\tau=X^{-1}+2\pi i Y$ where $|Y|\le \frac{1}{2\pi X}$, we have $|\Arg(\tau)|\le \frac{\pi}{4}$. Since for $\Re(s)=-\frac{3}{4}$,
	$$|\tau^{-s}|=\exp\left(\frac{3}{4}\log|\tau|+\Im(s)\Arg(\tau)\right)\le |\tau|^{\frac{3}{4}}e^{\frac{|\Im(s)|\pi}{4}},$$
	it follows that
	\begin{align*}
	|E_+|&=\left|\frac{1}{2\pi i}\int_{(-\tfrac{3}{4})} \Theta_+(s) ds\right|\\
	&=\left|\frac{1}{2\pi i}\int_{(-\tfrac{3}{4})} (4\tau)^{-s}\Gamma(s)\zeta(s+1)\Big(\zrie{s,\tfrac{1}{4}}-(1-2^{-s})\zrie{s,\tfrac{3}{4}}\Big) ds\right|\\
	&\le |\tau|^{\frac{3}{4}} \cdot \frac{1}{2\pi}\int_{-\infty}^\infty 4^{\frac{3}{4}} e^{\frac{|t|\pi}{4}}\left|\Gamma(-\tfrac{3}{4}+it)\right|\left|\zeta(\tfrac{1}{4}+it)\right|\\
	&\quad\times \Big(\left|\zrie{-\tfrac{3}{4}+it,\tfrac{1}{4}}\right|+(1+2^{\frac{3}{4}})\left|\zrie{-\tfrac{3}{4}+it,\tfrac{3}{4}}\right|\Big)dt\\
	&\le 0.507|\tau|^{\frac{3}{4}}.
	\end{align*}
	We also have
	$$|\tau|=\sqrt{X^{-2}+4\pi^2 Y^2}\le \sqrt{2}X^{-1}.$$
	Hence,
	$$|E_+|\le 0.66 X^{-\frac{3}{4}}.$$
	For $\log G(-e^{-\tau})$, we simply notice that
	\begin{align*}
	\log G(-e^{-\tau})=\frac{1}{2\pi i}\int_{(\tfrac{3}{2})} (4\tau)^{-s}\Gamma(s)\zeta(s+1)\Big(\zrie{s,\tfrac{3}{4}}-(1-2^{-s})\zrie{s,\tfrac{1}{4}}\Big)ds.
	\end{align*}
	The rest follows from similar calculations.
\end{proof}

\section{Applying the circle method}\label{sec:cir}

The proof of Theorem \ref{th:g(n)-asymp} is a simple exercise of the circle method. Let us first put
\begin{equation}
X=\sqrt{\frac{48 n}{\pi^2}}.
\end{equation}
Since it is assumed that $X\ge 3.4\times 10^{7}$ as in Theorem \ref{th:bound-away}, we get
\begin{equation}
n\ge2.4\times 10^{14}.
\end{equation}
Now applying Cauchy's integral formula gives
\begin{align}\label{eq:Cauchy}
g(n)&=\frac{1}{2\pi i}\int_{|q|=e^{-\frac{1}{X}}}\frac{G(q)}{q^{n+1}}dq\notag\\
&=e^{\frac{n}{X}}\int_{-\frac{1}{2\pi X}}^{1-\frac{1}{2\pi X}}G\big(e^{-(X^{-1}+2\pi i t)}\big)e^{2\pi i nt}\; dt.
\end{align}
From now on, we separate the interval $[-\frac{1}{2\pi X},\;1-\frac{1}{2\pi X}]$ into three disjoint subintervals:
\begin{align*}
I_1&:=\left[-\frac{1}{2\pi X},\;\frac{1}{2\pi X}\right],\\
I_2&:=\left[\frac{1}{2}-\frac{1}{2\pi X},\;\frac{1}{2}+\frac{1}{2\pi X}\right],\\
I_3&:=\left[-\frac{1}{2\pi X},\;1-\frac{1}{2\pi X}\right]-I_1-I_2.
\end{align*}
Before evaluating \eqref{eq:Cauchy} for each subinterval, we fix the notation that $\mfO(x)$ means an expression $E$ such that $|E|\le x$. We also write for $j=1,2,3$,
$$g_j(n):=e^{\frac{n}{X}}\int_{I_j} G\big(e^{-(X^{-1}+2\pi i t)}\big)e^{2\pi i nt}\; dt.$$

Let us begin with the evaluation of $g_1(n)$. Here,
\begin{align*}
g_1(n)&=e^{\frac{n}{X}}\int_{-\frac{1}{2\pi X}}^{\frac{1}{2\pi X}} G\big(e^{-(X^{-1}+2\pi i t)}\big)e^{2\pi i nt}\; dt\\
&=\frac{1}{2\pi i}\int_{\frac{1}{X}-i\frac{1}{X}}^{\frac{1}{X}+i\frac{1}{X}} e^{n\tau} G(e^{-\tau})\; d\tau.
\end{align*}
Notice that for $|x|\le 1$,
$$e^x = 1+\mfO(2|x|).$$
Applying \eqref{eq:near:q=1} yields
\begin{align}\label{eq:g1(n)-1}
g_1(n)=\frac{\big(1+\mfO(1.32X^{-\frac{3}{4}})\big)\Gamma(\tfrac{1}{4})}{2^{\frac{3}{4}}\pi^{\frac{1}{2}}}\frac{1}{2\pi i}\int_{\frac{1}{X}-i\frac{1}{X}}^{\frac{1}{X}+i\frac{1}{X}} \tau^{-\frac{1}{4}} \exp\left(\frac{\pi^2}{48}\frac{1}{\tau}+n\tau\right)\; d\tau.
\end{align}
We then separate the integral as
\begin{align*}
&\frac{1}{2\pi i}\int_{\frac{1}{X}-i\frac{1}{X}}^{\frac{1}{X}+i\frac{1}{X}} \tau^{-\frac{1}{4}} \exp\left(\frac{\pi^2}{48}\frac{1}{\tau}+n\tau\right)\; d\tau\\
&\quad = \frac{1}{2\pi i}\Bigg(\int_{\Gamma}-\int_{-\infty-i\frac{1}{X}}^{\frac{1}{X}-i\frac{1}{X}}+\int_{-\infty+i\frac{1}{X}}^{\frac{1}{X}+i\frac{1}{X}}\Bigg) \tau^{-\frac{1}{4}} \exp\left(\frac{\pi^2}{48}\frac{1}{\tau}+n\tau\right)\; d\tau\\
&\quad=: J_{11}+J_{12}+J_{13},
\end{align*}
where
\begin{align}\label{eq:Hankel contour}
\Gamma:=(-\infty-iX^{-1}) \to (X^{-1}-iX^{-1}) \to (X^{-1}+iX^{-1}) \to (-\infty+iX^{-1})
\end{align}
is a Hankel contour. To evaluate $J_{11}$, we make the change of variables
$$\tau=\sqrt{\frac{\pi^2}{48n}}w.$$
Then
\begin{align*}
J_{11}&=\left(\frac{\pi^2}{48n}\right)^{\frac{3}{8}}\frac{1}{2\pi i}\int_{\widetilde{\Gamma}} w^{-\frac{1}{4}}\exp\left(\sqrt{\frac{\pi^2 n}{48}}\left(\frac{1}{w}+w\right)\right)\; dw,
\end{align*}
where $\widetilde{\Gamma}$ is the new contour. Recalling the contour integral representation of $I_s(x)$:
$$I_s(x)=\frac{1}{2\pi i}\int_{\widetilde{\Gamma}} w^{-s-1}e^{\frac{x}{2}\left(w+\frac{1}{w}\right)}\ dw,$$
we conclude that
$$J_{11}=\frac{\pi^{\frac{3}{4}}}{2^{\frac{3}{2}}3^{\frac{3}{8}}n^{\frac{3}{8}}}I_{-\frac{3}{4}}\left(\frac{\pi}{2}\sqrt{\frac{n}{3}}\right).$$
To bound $J_{12}$, we put $\tau=x-i X^{-1}$. Then
\begin{align*}
J_{12}=\frac{1}{2\pi i}\int_{-\infty}^{X^{-1}} \tau^{-\frac{1}{4}} \exp\left(\frac{\pi^2}{48}\frac{1}{\tau}+n\tau\right)\; dx.
\end{align*}
Since $|\tau|\ge X^{-1}$, we have
$$|\tau|^{-\frac{1}{4}}\le X^{\frac{1}{4}}.$$
Also,
$$|e^{n\tau}|=e^{nx}.$$
Further,
$$\left|e^{\frac{\pi^2}{48}\frac{1}{\tau}}\right|=e^{\frac{\pi^2}{48}\frac{x}{x^2+X^{-2}}}\le e^{\frac{\pi^2}{96}X}.$$
Therefore,
\begin{align*}
|J_{12}|&\le \frac{1}{2\pi}\cdot X^{\frac{1}{4}}e^{\frac{\pi^2}{96}X}\int_{-\infty}^{X^{-1}} e^{nx}\; dx\\
&=\frac{1}{2\pi}\cdot X^{\frac{1}{4}}e^{\frac{\pi^2}{96}X}\cdot \frac{1}{n}e^{\frac{n}{X}}\\
&=\frac{3^{\frac{1}{8}}}{2^{\frac{1}{2}}\pi^{\frac{5}{4}}n^{\frac{7}{8}}}\exp\left(\frac{3\pi}{8}\sqrt{\frac{n}{3}}\right).
\end{align*}
One may carry out a similar argument and obtain that
$$|J_{13}|\le \frac{3^{\frac{1}{8}}}{2^{\frac{1}{2}}\pi^{\frac{5}{4}}n^{\frac{7}{8}}}\exp\left(\frac{3\pi}{8}\sqrt{\frac{n}{3}}\right).$$
Consequently,
\begin{align*}
\frac{1}{2\pi i}\int_{\frac{1}{X}-i\frac{1}{X}}^{\frac{1}{X}+i\frac{1}{X}} \tau^{-\frac{1}{4}} \exp\left(\frac{\pi^2}{48}\frac{1}{\tau}+n\tau\right)\; d\tau &= \frac{\pi^{\frac{3}{4}}}{2^{\frac{3}{2}}3^{\frac{3}{8}}n^{\frac{3}{8}}}I_{-\frac{3}{4}}\left(\frac{\pi}{2}\sqrt{\frac{n}{3}}\right)\\
&\quad+\mfO\Bigg(\frac{2^{\frac{1}{2}} 3^{\frac{1}{8}}}{\pi^{\frac{5}{4}}n^{\frac{7}{8}}}\exp\left(\frac{3\pi}{8}\sqrt{\frac{n}{3}}\right)\Bigg).
\end{align*}
Recalling \eqref{eq:g1(n)-1}, we have
\begin{align}\label{eq:g1(n)}
g_1(n)=\frac{\pi^{\frac{1}{4}}\Gamma(\tfrac{1}{4})}{2^{\frac{9}{4}}3^{\frac{3}{8}}n^{\frac{3}{8}}}I_{-\frac{3}{4}}\left(\frac{\pi}{2}\sqrt{\frac{n}{3}}\right)+E_{g_1},
\end{align}
where
\begin{align}\label{eq:E_g1}
|E_{g_1}|&\le \frac{\Gamma(\tfrac{1}{4})}{2^{\frac{3}{4}}\pi^{\frac{1}{2}}}\Bigg(\frac{1.32\pi^{\frac{3}{2}}}{2^{3}3^{\frac{3}{4}}n^{\frac{3}{4}}}I_{-\frac{3}{4}}\left(\frac{\pi}{2}\sqrt{\frac{n}{3}}\right)+\left(1+\frac{1.32\pi^{\frac{3}{4}}}{2^{\frac{3}{2}}3^{\frac{3}{8}}n^{\frac{3}{8}}}\right)\frac{2^{\frac{1}{2}} 3^{\frac{1}{8}}}{\pi^{\frac{5}{4}}n^{\frac{7}{8}}}\exp\left(\frac{3\pi}{8}\sqrt{\frac{n}{3}}\right)\Bigg)\notag\\
&\ll n^{-\frac{3}{4}}I_{-\frac{3}{4}}\left(\frac{\pi}{2}\sqrt{\frac{n}{3}}\right).
\end{align}

For $g_2(n)$, we have
\begin{align*}
g_2(n)&=(-1)^n e^{\frac{n}{X}}\int_{-\frac{1}{2\pi X}}^{\frac{1}{2\pi X}} G\big({-e}^{-(X^{-1}+2\pi i t)}\big)e^{2\pi i nt}\; dt\\
&=\frac{(-1)^n}{2\pi i}\int_{\frac{1}{X}-i\frac{1}{X}}^{\frac{1}{X}+i\frac{1}{X}} e^{n\tau} G({-e}^{-\tau})\; d\tau.
\end{align*}
It follows from \eqref{eq:near:q=-1} that
\begin{align}\label{eq:g2(n)-1}
g_2(n)=(-1)^n \frac{\big(1+\mfO(1.64X^{-\frac{3}{4}})\big)\Gamma(\tfrac{3}{4})}{2^{\frac{1}{4}}\pi^{\frac{1}{2}}}\frac{1}{2\pi i}\int_{\frac{1}{X}-i\frac{1}{X}}^{\frac{1}{X}+i\frac{1}{X}} \tau^{\frac{1}{4}} \exp\left(\frac{\pi^2}{48}\frac{1}{\tau}+n\tau\right)\; d\tau.
\end{align}
Similarly, we separate the integral as
\begin{align*}
&\frac{1}{2\pi i}\int_{\frac{1}{X}-i\frac{1}{X}}^{\frac{1}{X}+i\frac{1}{X}} \tau^{\frac{1}{4}} \exp\left(\frac{\pi^2}{48}\frac{1}{\tau}+n\tau\right)\; d\tau\\
&\quad = \frac{1}{2\pi i}\Bigg(\int_{\Gamma}-\int_{-\infty-i\frac{1}{X}}^{\frac{1}{X}-i\frac{1}{X}}+\int_{-\infty+i\frac{1}{X}}^{\frac{1}{X}+i\frac{1}{X}}\Bigg) \tau^{\frac{1}{4}} \exp\left(\frac{\pi^2}{48}\frac{1}{\tau}+n\tau\right)\; d\tau\\
&\quad=: J_{21}+J_{22}+J_{23},
\end{align*}
where the Hankel contour $\Gamma$ is as in \eqref{eq:Hankel contour}. In the same vein, one may compute that
$$J_{21}=\frac{\pi^{\frac{5}{4}}}{2^{\frac{5}{2}}3^{\frac{5}{8}}n^{\frac{5}{8}}}I_{-\frac{5}{4}}\left(\frac{\pi}{2}\sqrt{\frac{n}{3}}\right).$$
To bound $J_{22}$, we still put $\tau=x-i X^{-1}$. Noticing that
$$|\tau|^{\frac{1}{4}}=(x^2+X^{-2})^{\frac{1}{8}}\le |x|^{\frac{1}{4}}+X^{-\frac{1}{4}},$$
we have
\begin{align*}
|J_{22}|&\le \frac{1}{2\pi}\cdot e^{\frac{\pi^2}{96}X}\int_{-\infty}^{X^{-1}} e^{nx}\big(|x|^{\frac{1}{4}}+X^{-\frac{1}{4}}\big)\; dx\\
&\le \frac{1}{2\pi}\cdot e^{\frac{\pi^2}{96}X}\int_{-\infty}^{0} e^{nx}(-x)^{\frac{1}{4}}\; dx+\frac{1}{2\pi}\cdot e^{\frac{\pi^2}{96}X}\int_{-\infty}^{X^{-1}} e^{nx}\cdot 2X^{-\frac{1}{4}}\; dx\\
&=\frac{\Gamma(\tfrac{5}{4})}{2 \pi n^{\frac{5}{4}}}\exp\left(\frac{\pi}{8}\sqrt{\frac{n}{3}}\right)+\frac{1}{2^{\frac{1}{2}}3^{\frac{1}{8}}\pi^{\frac{3}{4}}n^{\frac{9}{8}}}\exp\left(\frac{3\pi}{8}\sqrt{\frac{n}{3}}\right).
\end{align*}
Likewise,
$$|J_{23}|\le \frac{\Gamma(\tfrac{5}{4})}{2 \pi n^{\frac{5}{4}}}\exp\left(\frac{\pi}{8}\sqrt{\frac{n}{3}}\right)+\frac{1}{2^{\frac{1}{2}}3^{\frac{1}{8}}\pi^{\frac{3}{4}}n^{\frac{9}{8}}}\exp\left(\frac{3\pi}{8}\sqrt{\frac{n}{3}}\right).$$
In consequence,
\begin{align}\label{eq:g2(n)}
g_2(n)=(-1)^n\frac{\pi^{\frac{3}{4}}\Gamma(\tfrac{3}{4})}{2^{\frac{11}{4}}3^{\frac{5}{8}}n^{\frac{5}{8}}}I_{-\frac{5}{4}}\left(\frac{\pi}{2}\sqrt{\frac{n}{3}}\right)+E_{g_2},
\end{align}
where
\begin{align}
|E_{g_2}|&\le \frac{\Gamma(\tfrac{3}{4})}{2^{\frac{1}{4}}\pi^{\frac{1}{2}}}\Bigg(\frac{1.64\pi^{2}}{2^{4}3^1 n}I_{-\frac{5}{4}}\left(\frac{\pi}{2}\sqrt{\frac{n}{3}}\right)\notag\\
&\quad+2\left(1+\frac{1.64\pi^{\frac{3}{4}}}{2^{\frac{3}{2}}3^{\frac{3}{8}}n^{\frac{3}{8}}}\right) \left(\frac{\Gamma(\tfrac{5}{4})}{2 \pi n^{\frac{5}{4}}}\exp\left(\frac{\pi}{8}\sqrt{\frac{n}{3}}\right)+\frac{1}{2^{\frac{1}{2}}3^{\frac{1}{8}}\pi^{\frac{3}{4}}n^{\frac{9}{8}}}\exp\left(\frac{3\pi}{8}\sqrt{\frac{n}{3}}\right)\right)\Bigg)\notag\\
&\ll n^{-1}I_{-\frac{5}{4}}\left(\frac{\pi}{2}\sqrt{\frac{n}{3}}\right).
\end{align}

\begin{remark}
It is necessary to point out that $g_2(n)$ has an absolute size of
$$\textsf{constant}\times n^{-\frac{5}{8}}I_{-\frac{5}{4}}\left(\frac{\pi}{2}\sqrt{\frac{n}{3}}\right),$$
while from \eqref{eq:E_g1},
$$E_{g_1}\ll n^{-\frac{3}{4}}I_{-\frac{3}{4}}\left(\frac{\pi}{2}\sqrt{\frac{n}{3}}\right).$$
Since the two $I$-Bessel functions have the same order, we conclude that $E_{g_1}$ is negligible in comparison with $g_2(n)$.
\end{remark}

Finally,
\begin{align*}
g_3(n)=e^{\frac{n}{X}}\int_{I_3} G\big(e^{-(X^{-1}+2\pi i t)}\big)e^{2\pi i nt}\; dt.
\end{align*}
Hence, by Theorem \ref{th:bound-away},
\begin{align*}
|g_3(n)|&\le e^{\frac{n}{X}}\int_{I_3} \exp\Bigg(\left(\frac{\pi^2}{48}-\frac{1}{100}\right)X\Bigg)\; dt\\
&\le \exp\Bigg(\frac{n}{X}+\left(\frac{\pi^2}{48}-\frac{1}{100}\right)X\Bigg).
\end{align*}
Namely,
\begin{align}\label{eq:g3(n)}
|g_3(n)|&\le \exp\left(\frac{\pi}{2}\sqrt{\frac{n}{3}}-\frac{\sqrt{3n}}{25\pi}\right).
\end{align}

The asymptotic formula \eqref{eq:g(n)-asymp} follows from \eqref{eq:g1(n)}, \eqref{eq:g2(n)} and \eqref{eq:g3(n)}. Further, a simple calculation reveals that when $n\ge2.4\times 10^{14}$, the sign of $g(n)$ depends only on the leading term
$$\frac{\pi^{\frac{1}{4}}\Gamma(\tfrac{1}{4})}{2^{\frac{9}{4}}3^{\frac{3}{8}}n^{\frac{3}{8}}}I_{-\frac{3}{4}}\left(\frac{\pi}{2}\sqrt{\frac{n}{3}}\right),$$
which is of course positive.

\subsection*{Acknowledgements}

I would like to express my gratitude to Ae Ja Yee for introducing Conjecture \ref{conj-sy} at the Combinatorics/Partitions Seminar of Penn State, and to George Andrews for our fruitful weekly discussions. Thanks also go to the referees whose detailed comments were extremely helpful in improving the exposition of this paper.

\bibliographystyle{amsplain}

\end{document}